\documentclass[12pt,reqno]{amsart}
\usepackage{amssymb}
\usepackage[latin1]{inputenc}
\usepackage{dsfont}
\usepackage{hyperref}
\usepackage{fullpage}
\newtheorem{theorem}{Theorem}[section]

\newtheorem{proposition}[theorem]{Proposition}
\newtheorem{corollary}[theorem]{Corollary}
\theoremstyle{definition}
\newtheorem{definition}[theorem]{Definition}

\theoremstyle{remark}
\newtheorem{remark}[theorem]{Remark}
\numberwithin{equation}{section}
\begin{document}
\title{On a flow of operators associated to virtual permutations}
\author[J. Najnudel]{Joseph Najnudel}
\address{Institut f\"ur Mathematik, Universit\"at Z\"urich, Winterthurerstrasse 190,
8057-Z\"urich, Switzerland}
\email{\href{mailto:joseph.najnudel@math.uzh.ch}{joseph.najnudel@math.uzh.ch}}
\author[A. Nikeghbali]{Ashkan Nikeghbali}
\email{\href{mailto:ashkan.nikeghbali@math.uzh.ch}{ashkan.nikeghbali@math.uzh.ch}}
\date{\today}
\begin{abstract}
In \cite{KOV}, Kerov, Olshanski and Vershik introduce the so-called virtual permutations, 
defined as families of permutations $(\sigma_N)_{N \geq 1}$, $\sigma_N$ in the symmetric group of order
$N$, such that the cycle structure of $\sigma_N$ can be deduced from the structure of $\sigma_{N+1}$ 
simply by removing the element $N+1$. The virtual permutations, and in particular 
the probability measures on the corresponding space which are invariant by conjugation, have been studied in a 
more detailed way by Tsilevich in \cite{Tsi} and \cite{Tsi98}. In the present article, we prove that for a large class of such
 invariant measures (containing in particular the Ewens measure of any parameter $\theta \geq 0$), 
it is possible to associate a flow $(T^{\alpha})_{\alpha \in \mathbb{R}}$
of random operators on a suitable functional space. Moreover, if $(\sigma_N)_{N \geq 1}$ is a random virtual permutation following a
distribution in the class described above,  
the operator $T^{\alpha}$ can be interpreted as the limit, in a sense which
has to be made precise, of the permutation $\sigma_N^{\alpha_N}$, where $N$ goes to infinity and 
$\alpha_N$ is equivalent to $\alpha N$. In relation with this interpretation, we prove that the 
eigenvalues of the infinitesimal generator of $(T^{\alpha})_{\alpha \in \mathbb{R}}$ 
are equal to the limit of the rescaled eigenangles of the permutation matrix 
associated to $\sigma_N$. 
\end{abstract}
\maketitle
\section{Introduction}
A large part of the research in random matrix theory comes from the problem of finding
the possible limit distributions of the eigenvalues of a given ensemble of random matrices, when the
 dimension goes to infinity. Many different ensembles have been studied (see, for example, Mehta \cite{Meh}), the most classical one
is the so-called {\it Gaussian Unitary Ensemble} (GUE), where the corresponding random matrix is hermitian, the 
diagonal entries are standard gaussian variables and the real and the imaginary part of the entries above 
the diagonal are centered gaussian variables of variance $1/2$, all these variables being independent.
Another well-known ensemble is the {\it Circular Unitary Ensemble} (CUE), corresponding to a random 
matrix which follows the Haar measure on a finite-dimensional unitary group. A remarkable
phenomenon which happens is the so-called {\it universality property}: there are some particular
point processes which appear as the limit (after suitable scaling) of the distribution of 
the eigenvalues (or eigenangles) associated with a large class of random matrix ensembles, this
limit being independent of the detail of the model which is considered. 
 For example, after scaling, the small eigenvalues of the GUE and the small eigenangles
of the CUE converge to the same process, called {\it determinantal process
with sine kernel}, and appearing  as the limit of a number of other models of random matrices. This
 process is a point process on the real line, such that informally, for $x_1, \dots, x_n \in \mathbb{R}$,
the probability to have a point in the neighborhood of $x_j$ for all $j \in \{1, \dots, k\}$
is proportional to the determinant of the matrix $( K(x_j,x_k))_{1 \leq 
j, k \leq n}$, where the kernel $K$ is given by the formula:
$$ K(x,y) = \frac{\sin (\pi (x-y))}{\pi (x-y)}.$$
Another point process which enjoys some universality properties is the {\it determinantal process
with Airy kernel}, which is involved in the distribution of the largest eigenvalues 
of the GUE, and which is defined similarly as the determinantal process
with sine kernel, except that the kernel $K$ is now given by:
$$ K(x,y) = \frac{ \operatorname{Ai} (x) \operatorname{Ai}'(y) -
\operatorname{Ai} (y) \operatorname{Ai}'(x)}{x-y},$$
where $\operatorname{Ai}$ denotes the Airy function.
 The phenomenon of universality is not well-explained in its full generality: a possible
way the have a good understanding of the corresponding point processes is to express them as the set
 of eigenvalues of some universal infinite-dimensional
 random operators, and to prove that these operators are the limits, in a sense
 which needs to be made precise, of the classical random matrix ensembles.
Such a construction has been done by Ramirez, Rider and Vir\'ag \cite{RVV},
where, as a particular case, the authors express a determinantal process with Airy kernel
as the set of the eigenvalues of a random differential operator on a functional space. 
Moreover, this infinite-dimensional operator is naturally interpreted as the limit of an ensemble of tridiagonal matrices, 
which have the same eigenvalue distribution as the GUE. However, 
this operator is not directly constructed from GUE (or another classical model as CUE), and the reason
of its universal properties are not obvious. On the other hand, an operator whose eigenvalues 
form a determinantal process with sine kernel, and which is related in a natural way to a classical random matrix model,
 has not been yet defined: a construction 
which seems to be very promising involves the interpretation of the sine kernel process
 as a function of a stochastic process called the Brownian carousel
and constructed by Valk\'o and Vir\'ag \cite{VV}. Now, despite the fact that it seems to be particularly difficult to 
associate, in a natural way, an infinite-dimensional operator to the most classical matrix ensembles, we shall prove, in the present 
paper, that such a construction is possible and very explicit for a large class of ensembles of permutations matrices. 
These ensembles, and some of their generalizations, have already be studied by several authors (see, for example, Evans \cite{Ev} and Wieand \cite{Wie}), including 
the authors of the present paper (see \cite{NNp}). An important advantage of permutations matrices is the fact that 
their eigenvalues can directly be expressed in function of the size of the cycles of the corresponding permutations. Hence, it is equivalent 
to study these matrices or to deal with the corresponding cycle structure. Another advantage is the existence of
a quite convenient way to define models of permutation matrices in all the
different dimensions, on the same probability space, which gives a meaning to the notion of almost sure convergence when the dimension goes 
to infinity (a rather unusual situation in random matrix theory). This also gives the possibility to define an infinite-dimensional limit model, which will 
be explicitly constructed in this paper. Once the permutations and their matrices are identified, the main objects involved in 
our construction are the so-called {\it virtual permutations}, first introduced by Kerov, Olshanski and Vershik in \cite{KOV},
and further  studied  by Tsilevich in \cite{Tsi} and \cite{Tsi98}. A virtual permutation can be defined as 
follows: it is a sequence $(\sigma_N)_{N \geq 1}$ of permutations, $\sigma_N$ being of order $N$, 
and such that the cycle structure of $\sigma_N$ is obtained from the cycle structure of $\sigma_{N+1}$ 
by simply  removing the element $N+1$. Then, for all $\theta \geq 0$, it is possible to define a unique
probability measure on the space of virtual permutations, such that for all $N \geq 0$, its image 
by the $N$-th coordinate is equal to the Ewens measure on the symmetric group of order $N$, with 
parameter $\theta$. The Ewens measures on the space of virtual permutations are particular cases of the so-called {\it central measures}, studied in \cite{Tsi98}, 
and defined as
the probability measures which are invariant by conjugation with any permutation of finite order (see below for the details).
The central measures are completely characterized in \cite{Tsi98}, using the properties of exchangeable partitions, stated
by Kingman in \cite{King75}, \cite{King78a} and \cite{King78b} (see also the course by Pitman \cite{Pit}). 
The following remarkable statement holds: if the distribution of a random virtual permutation $(\sigma_N)_{N \geq 1}$ is 
a central measure, and if for all $N \geq 1$, $(l^{(N)}_k)_{k \geq 1}$ denotes the decreasing sequence of cycle 
lengths (completed by zeros) of the permutation $\sigma_N$,
then for all $k \geq 1$, 
the sequence $(l^{(N)}_k/N)_{N \geq 1}$ converges almost surely to a limit random variable 
$\lambda_k$. By using Fatou's lemma, one immediately deduces that almost surely, the random sequence $(\lambda_k)_{k \geq 1}$ is 
nonnegative and decreasing, with:
$$\sum_{k \geq 1} \lambda_k \leq 1.$$
If the distribution of $(\sigma_N)_{N \geq 1}$ is the Ewens measure of parameter $\theta \geq 0$, then
$(\lambda_k)_{k \geq 1}$ is a Poisson-Dirichlet process with parameter $\theta$ (if $\theta = 0$, then $\lambda_1 = 1$ and 
$\lambda_k = 0$ for $k \geq 1$). 
This convergence of the renormalized cycle lengths can be translated into a statement on random matrices. 
Indeed, to any virtual permutation following a central measure, one can associate a sequence 
$(M_N)_{N \geq 1}$ of random permutation
matrices, $M_N$ being of dimension $N$. If for $N \geq 1$, $X_N$ denotes 
the point process of the eigenangles of $M_N$, multiplied by $N$ and counted with multiplicity,
then $X_N$ converges almost surely to the limit process $X_{\infty}$ defined as follows: 
\begin{itemize}
\item $X_{\infty}$ contains, for all $k \geq 1$ such that $\lambda_k > 0$, each non-zero multiple of $2\pi/ \lambda_k$.
\item The multiplicity of any non-zero point $x$ of $X_{\infty}$ is equal to the number of values of $k$ such that $\lambda_k > 0$ and 
$x$ is multiple of $2 \pi/\lambda_k$.
\item The multiplicity of zero is equal to the number of values of $k \geq 1$ such that $\lambda_k > 0$ if 
$\sum_{k \geq 1} \lambda_k = 1$, and to infinity if $\sum_{k \geq 1} \lambda_k < 1$. 
\end{itemize} 
This convergence has to be understood in the following way: for all functions $f$ from $\mathbb{R}$ to 
$\mathbb{R}_+$, continuous with compact support, the sum of $f$ at the points of $X_N$, counted with 
multiplicity, tends to the corresponding sum for $X_{\infty}$ when $N$ goes to infinity. 
In the case of the Ewens probability measure with parameter $\theta > 0$, a detailed proof of this result is given by the present authors
in \cite{NNp}, in
 a more general context. Once the almost sure convergence of the rescaled
eigenangles is established, one can naturally ask the following question: is it possible to express
the limit point process as the spectrum of a random operator associated to the virtual permutation
which is considered? The main goal of this article is to show that the answer is positive. More precisely,
we prove that in a sense which can be made precise, for almost every virtual permutation following a central probability 
distribution, for all $\alpha \in \mathbb{R}$, and 
for all sequences $(\alpha_N)_{N \geq 1}$ such that $\alpha_N$ is equivalent to $\alpha N$ for
$N$ going to infinity, $\sigma_N^{\alpha_N}$ converges to an operator $T^{\alpha}$, depending 
 on $(\sigma_N)_{N \geq 0}$, on $\alpha$ but not on the choice of $(\alpha_N)_{N \geq 1}$. 
The flow of operators $(T^{\alpha})_{\alpha \in \mathbb{R}}$ (which almost surely satisfies 
$T^{\alpha + \beta} = T^{\alpha} T^{\beta}$ for all $\alpha, \beta \in \mathbb{R}$) is defined 
on a functional space which will be constructed later, and admits almost surely an infinitesimal 
generator $U$. Moreover, we prove that the spectrum of $iU$ is exactly given by the limit point 
process $X_{\infty}$ constructed above. This spectral interpretation of such a limit point 
process suggests that it is perhaps possible to construct similar objects for more classical matrix models, despite the discussion
above about the 
difficulty of this problem. Indeed, in \cite{OV}, Olshanski and 
Vershik characterize the measures on the space of infinite dimensional hermitian matrices, which are invariant by conjugation with finite-dimensional 
unitary matrices. These central measures enjoy the following property: if a random infinite matrix follows one of them, then 
after suitable scaling, the point process of the extreme eigenvalues of its upper-left finite-dimensional 
submatrice converges almost surely to 
a limit point process when the dimension of this submatrice tends to infinity, similarly as for the case of virtual permutations, where 
$X_N$ tends to $X_{\infty}$.
By the Cayley transform, the sequence of the upper-left submatrices of an infinite hermitian matrix is mapped
to a sequence $(M_N)_{N \geq 1}$ of unitary matrices, such that the matrix $M_N$ can be deduces from $M_{N+1}$ by 
an  projective map, explicitly described by Neretin \cite{Ner}. Among the central measures described in \cite{OV}, 
there exists a unique measure for which all its projections on the finite-dimensional unitary groups are equal to Haar measure. 
In \cite{BO} Borodin and Olshanski study the hermitian version of this measure, and a family of generalizations, depending on a complex parameter 
and called Hua-Pickrell measures. In particular, they prove that the corresponding point process is determantal and 
compute explicitly its kernel: for the Haar measure, they obtain the image of a sine kernel process by the map 
$x \mapsto 1/x$. Moreover, in a forthcoming paper with Bourgade, by using the splitting of unitary matrices as a product 
of reflections, we define the natural generalization 
of virtual permutations to sequences of general unitary matrices of increasing dimensions. These sequences are different
from the sequences of unitary matrices defined by Neretin, Olshanski and Vershik, in spite  of the similarity between 
the two constructions. In our setting, there exists also a projective limit of the Haar measure, and under this measure, we prove the 
almost sure convergence of the rescaled eigenangles when the dimension goes to infinity, the limit process being a sine-kernel process. 
  It is also possible to construct a family of measures
which has the same finite-dimensional projections as the Hua-Pickrell measures, and to study the almost sure convergence of their rescaled eigenangles. 
However, contrarily to the case of virtual permutations discussed in the present paper, we are not yet able to interpret the corresponding limit 
point process as the sequence of eigenvalues of an operator. 

The present article is organized as follows: in Section \ref{s1}, we define the notion of a virtual 
permutation in a more general setting than it is usually done; in Section 
\ref{s2}, we study the central measures in this new setting, generalizing the results by Kerov, Olshanski, 
Tsilevich and Vershik; in Section \ref{s3}, we use the results of Section \ref{s2} in order to associate, to almost every virtual permutation under a central 
measure, a flow of 
transformations of a topological space obtained by completing the set on which the virtual permutations 
act;  in Section \ref{s4}, we interpret this flow as a flow of operators on a functional space, and
we deduce the construction of the random operator $U$. 
\section{Virtual permutations of general sets} \label{s1}
The virtual permutations are usually defined as the sequences $(\sigma_N)_{N \geq 1}$
such that $\sigma_N \in \Sigma_N$ for all $N \geq 1$, where $\Sigma_N$ is the 
symmetric group of order $N$, and the cycle structure of $\sigma_N$ is obtained by removing $N+1$ from 
the cycle structure of $\sigma_{N+1}$. In this definition, the order of the integers is involved in an 
important way, which is not very satisfactory since we are essentially interested in the cycle structure 
of the permutation, and not particularly in the nature of the elements inside the cycles. Therefore, in this
 section, we present a notion of virtual permutation which can be applied to any set and not only 
to the set of positive integers. This generalization is possible because of  following result:
\begin{proposition} \label{p1}
Let $(\sigma_N)_{N \geq 1}$ be a virtual permutation (in the usual sense). For all finite subsets
$I \subset \mathbb{N}^*$, and for all $N \geq 1$ larger than any of the elements in $I$, 
let $\sigma^{(N)}_I$ be the permutation of the elements of $I$ obtained by removing the
elements outside $I$ from the cycle structure of $\sigma_N$. Then $\sigma^{(N)}_I$ depends only on
$I$ and not on the choice of $N$ majorizing $I$, and one can write $\sigma^{(N)}_I =: \sigma_I$.
Moreover, if $J$ is a finite subset of $\mathbb{N}^*$ containing $I$, then $\sigma_I$ can be obtained
from the cycle structure of $\sigma_J$ by removing all the elements of $J \backslash I$. 
\end{proposition}
\begin{proof}
Let $N \leq N'$ be two integers majorizing $I$. By the classical definition of virtual permutations,
$\sigma_N$ is obtained by removing the elements strictly larger than $N$ from the cycle 
structure of $\sigma_{N'}$. Hence, $\sigma^{(N)}_I$ can be obtained from the cycle structure
of $\sigma_{N'}$ by removing the elements strictly larger than $N$, and then the elements smaller than 
or equal to $N$ which are not in $I$. Since all the elements of $I$ are smaller than or equal to $N$, 
it is equivalent to remove directly all the elements of $\{1,\dots,N'\}$ outside $I$, which proves 
that $\sigma^{(N')}_I = \sigma^{(N)}_I$. Now, let $J$ be a finite subset of $\mathbb{N}^*$ containing $I$,
and $N''$ an integer which majorizes $J$. The permutation $\sigma_I$ is obtained by removing
the elements of $\{1,\dots,N''\} \backslash I$ from the cycle structure of $\sigma_{N''}$. 
It is equivalent to say that $\sigma_I$ is obtained from $\sigma_{N''}$ by  removing
 the elements of $\{1,\dots,N''\} \backslash J$, and then the elements of $J \backslash I$, which 
implies that $\sigma_I$ can be obtained by removing the elements of $J \backslash I$ from the
 cycle structure of $\sigma_J$. 
\end{proof}
\noindent
We can then define virtual permutations on general sets as follows:
\begin{definition} \label{d1}
 A virtual permutation of a given set $E$ is a family of permutations $(\sigma_I)_{I \in 
\mathcal{F} (E)}$, indexed by the set $\mathcal{F}(E)$ of the finite subsets of $E$, such that 
$\sigma_I \in \Sigma_I$, where $\Sigma_I$ is the symmetric group of $I$, and such that for 
all $I, J \in \mathcal{F}(E)$, $I \subset J$, the permutation $\sigma_I$ is obtained by removing 
the elements of $J \backslash I$ from the cycle structure of $\sigma_J$.
\end{definition}
\noindent
If $E$ is a finite set, a virtual permutation is essentially a permutation, more precisely, one 
has the following proposition:
\begin{proposition}  \label{viper}
Let $E$ be a finite set and $\sigma$ a permutation of $E$. Then, one can define a
virtual permutation $(\sigma_I)_{I \in \mathcal{F}(E)}$ of $E$ as follows: $\sigma_I$ is obtained 
from $\sigma$ by removing the elements of $E \backslash I$ from its cycle structure. Moreover, 
this mapping is bijective, and the inverse mapping is obtained by associating the permutation 
$\sigma_E$ to any virtual permutation $(\sigma_I)_{I \in \mathcal{F}(E)}$.
\end{proposition}
\noindent
This proposition is almost trivial, so we omit the proof. 

\noindent\textbf{Notation} Because of this result, we will 
denote the set of virtual permutations of any set $E$ by $\Sigma_E$. In
 the case $E = \mathbb{N}^*$, we can immediately deduce the following result from Proposition \ref{p1}: 
\begin{proposition}
If $(\sigma_N)_{N \geq 1}$ is a virtual permutation in the classical sense, then with the notation
of Proposition \ref{p1}, $(\sigma_I)_{I \in \mathcal{F}(\mathbb{N}^*)}$ is a virtual permutation 
of $\mathbb{N}^*$ in the sense of Definition \ref{d1}. Moreover, the mapping which associates
$(\sigma_I)_{I \in \mathcal{F}(\mathbb{N}^*)}$ to $(\sigma_N)_{N \geq 1}$ is bijective:
the existence of the inverse mapping is deduced from the fact that
 $\sigma_N = \sigma_{\{1, \dots, N \}}$. 
\end{proposition}
\noindent
An example of a virtual permutation defined on an uncountable set is given as follows:
if $E$ is a circle, and if for $I \in \mathcal{F}(E)$, $\sigma_I$ is the permutation 
of $I$ containing a unique cycle, obtained by counterclockwise enumeration of the points
of $I$, then $(\sigma_I)_{I \in \mathcal{F}(E)}$ is a virtual permutation of $E$.
In this example, it is natural to say that $(\sigma_I)_{I \in \mathcal{F}(E)}$ contains $E$ itself as a unique 
cycle. More generally, one can define a cycle structure for any virtual permutation, by the following result:
\begin{proposition} \label{equivalence}
Let $(\sigma_I)_{I \in \mathcal{F}(E)}$ be a virtual permutation of a set $E$, and let 
$x, y$ be two elements of $E$. Then, one of the  following two possibilities holds:
\begin{itemize} 
\item For all $I \in \mathcal{F}(E)$ containing $x$ and $y$, these two elements are in the same 
cycle of the permutation $\sigma_I$;
\item For all $I \in \mathcal{F}(E)$ containing $x$ and $y$, these two elements are in two different 
cycles of $\sigma_I$.
\end{itemize}
The first case defines an equivalence relation on the set $E$. 
\end{proposition}
\begin{proof}
Let $I \in \mathcal{F}(E)$ containing $x$ and $y$. By the definition of virtual permutations, it is 
easy to check that $x$ and $y$ are in the
same cycle of $\sigma_I$ if and only if $\sigma_{\{x,y\}}$ is equal to the transposition $(x,y)$, these
property being independent of the choice of $I$. Moreover, if the first item in Proposition
 \ref{equivalence} holds for $x$ and $y$, and for $y$ and $z$, then the permutation
$\sigma_{\{x,y,z\}}$ contains a unique cycle, which implies that the first item holds also for 
$x$ and $z$.
\end{proof}
\noindent
From now, the equivalence relation defined in Proposition \ref{equivalence} will be denoted 
by $\sim_{(\sigma_I)_{I \in \mathcal{F}(E)}}$, or simply by $\sim$ if no confusion is possible. 
The corresponding equivalence classes will be called the {\it cycles} of $(\sigma_I)_{I \in \mathcal{F}(E)}$.
The cycle of an element $x \in E$ will be denoted $\mathcal{C}_{(\sigma_I)_{I \in \mathcal{F}(E)}} (x)$,
or simply $\mathcal{C}(x)$. One immediately checks that this notion of cycle is consistent with the classical 
notion for permutations of finite order. Another notion which should be introduced is the notion of conjugation, 
 involved in the definition of central measures given in Section \ref{s2}.
\begin{proposition} \label{conj}
Let $\Sigma^{(0)}_E$ be the group of 
the permutations of $E$ which fix all but finitely many elements
 of $E$. Then, the group $\Sigma^{(0)}_E$ acts on $\Sigma_E$ by conjugation, because of the following fact: for all $g \in \Sigma^{(0)}_E$, and
for all $\sigma = (\sigma_I)_{I \in \mathcal{F}(E)}$, there exists
  a unique virtual permutation $g \sigma g^{-1} = (\sigma'_I)_{I \in \mathcal{F}(E)}$
 such that for all $I \in \mathcal{F}(E)$ containing all the points of $E$ which are not fixed by $g$, $\sigma'_I = g_I \sigma_I g_I^{-1}$, where $g_I$ denotes 
 the restriction of $g$ to $I$.  Moreover, the cycles of $g \sigma g^{-1}$ are the images by $g$ of the cycles of $\sigma$. 
\end{proposition}
\begin{proof}
Let $\sigma = (\sigma_I)_{I \in \mathcal{F}(E)} \in \Sigma_E$, $g \in \Sigma^{(0)}_E$ and denote $E(g)$ the set of points 
which are not fixed by $g$. For all $I \in \mathcal{F}(E)$, let us define $\sigma'_I$ as the permutation obtained from $g_{E(g) \cup I} \sigma_{E(g) \cup I} 
g_{E(g) \cup I}^{-1}$ by removing the elements of $E(g) \backslash I$ from its cycle structure. In other words, the cycle structure of $\sigma'_I$ can be obtained 
from the cycle structure of $\sigma_{E(g) \cup I}$ by replacing all the elements by their image by $g$, and then by removing the elements 
of $E(g) \backslash I$. Now, if $J \in \mathcal{F}(E)$ and $I \subset J$, then the structure of $\sigma'_I$ can also be
obtained from the structure of $\sigma_{E(g) \cup J}$ by removing the elements of $J \backslash ({E(g) \cup I})$, by replacing the remaining 
elements by their image by $g$, and then by removing the elements 
of $E(g) \backslash I$. Since all the elements of $J \backslash {E(g) \cup I}$ are fixed by $g$, the order of the two first operations is not important, and 
then $\sigma'_I$ is obtained from $\sigma_{E(g) \cup J}$ by replacing the elements by their image by $g$, and then by removing the elements 
of $(J \cup E(g)) \backslash I$. This implies that $\sigma'_I$ is obtained from $\sigma'_J$ by removing the elements of $J \backslash I$ from 
its cycle structure, in other words, $(\sigma'_I)_{I \in \mathcal{F}(E)}$ is a virtual permutation of $E$, which proves the existence of $g \sigma g^{-1}$. 
Its uniqueness is a direct consequence of the fact that all its components corresponding to a set containing $E(g)$ are determined by definition. 
Since, for $I$ containing $E(g)$, the cycle structure of $(g \sigma g^{-1})_{I}$ is obtained from the structure of $\sigma$ by replacing the 
elements by their image by $g$, it is easy to deduce that  the cycles of the virtual permutation $g \sigma g^{-1}$ are the images by $g$ of the cycles of $\sigma$, and 
that the conjugation is a group action of $\Sigma^{(0)}_E$ on $\Sigma_E$.
\end{proof}
\noindent
We have so far given the general construction of virtual permutations, and some of their main properties: in Section \ref{s2}, we introduce probability measures on the space of virtual permutations. 
\section{Central measures on general spaces of virtual permutations} \label{s2}
In order to construct a probability measure on the space of virtual permutations of a set $E$, one can expect that 
it is sufficient to define its images by the coordinate mappings, if these images satisfy some compatibility properties. 
The following result proves that such a construction is possible if the set $E$ is countable:
\begin{proposition} \label{ewensvi}
Let $E$ be a countable set, let $S(E)$ be the product of all the symmetric groups 
$\Sigma_I$ for $I \in \mathcal{F}(E)$, and let $\mathcal{S}(E)$ be the $\sigma$-algebra on $S(E)$,
 generated by all the coordinates mappings $(\sigma_I)_{I \in \mathcal{F}(E)}
\mapsto \sigma_J$ for $J \in \mathcal{F}(E)$. For all $I, J \in \mathcal{F}(E)$ such that $I \subset J$, let $\pi_{J,I}$ be 
the map from $\Sigma_J$ to $\Sigma_I$ which removes the elements of $J \backslash I$ from the cycle structure. 
Let $(\mathbb{P}_I)_{I \in \mathcal{F}(E)}$ be a family of probability measures, $\mathbb{P}_I$ defined on the finite set
$\Sigma_I$, which is compatible in the following sense: for all $I, J \in \mathcal{F}(E)$ such that $I \subset J$, the image of 
$\mathbb{P}_J$ by $\pi_{J,I}$ is $\mathbb{P}_I$. Then, there exists a unique probability measure 
$\mathbb{P}$ on the measurable space $(S(E),\mathcal{S}(E))$ satisfying
 the  following two conditions:
\begin{itemize}
\item $\mathbb{P}$ is supported by the set $\Sigma_E$ of virtual permutations;
\item For all $I \in \mathcal{F}(E)$, the image of $\mathbb{P}$ 
by the coordinate map  indexed by $I$ is equal to the measure $\mathbb{P}_I$. 
\end{itemize}
\end{proposition}
\begin{proof}
For every family $(I_k)_{1 \leq k \leq n}$ of elements of $\mathcal{F}(E)$, and for 
all $J \in \mathcal{F}(E)$ containing $I_k$ for all $k$, let $\mathbb{P}_{I_1,\dots,I_n;J}$ be the 
image of the probability measure  $\mathbb{P}_J$ by the mapping
$$\sigma \mapsto (\pi_{J,I_1}(\sigma),\dots,\pi_{J,I_n}(\sigma)),$$
from $\Sigma_J$ to the product space $\Sigma_{I_1} \times \dots \times \Sigma_{I_n}$.
Then, the  following two properties hold:
\begin{itemize}
\item Once the sets $(I_k)_{1 \leq k \leq n}$ are fixed, the measure 
$$\mathbb{P}_{I_1,\dots,I_n} := \mathbb{P}_{I_1,\dots,I_n; J}$$
does not depend on the choice of $J$.
\item For $n \geq 2$, the image of $\mathbb{P}_{I_1,\dots,I_n}$
by a permutation $\sigma \in \Sigma_n$ of the coordinates is equal to 
$\mathbb{P}_{I_{\sigma(1)},\dots,I_{\sigma(n)}}$, and the image by the application removing
the last coordinate is $\mathbb{P}_{I_1,\dots,I_{n-1}}$.
\end{itemize}
\noindent
These properties can easily be proven by using the compatibility property of the family $(\mathbb{P}_I)_{I \in \mathcal{F}(E)}$.
 Now, let us observe that a probability measure $\mathbb{P}$ on $(S(E),\mathcal{S}(E))$
 satisfies the conditions given in Proposition \ref{ewensvi} if and only if for all 
$I_1, \dots, I_n \in \mathcal{F}(E)$, the image of $\mathbb{P}$ by the family of coordinates 
indexed by $I_1, \dots, I_n$ is equal to $\mathbb{P}_{I_1,\dots,I_n}$ (note that the fact that $E$ is countable is used in this step). 
 The property of compatibility given in the second item above and the Caratheodory extension theorem then imply the existence
and the uniqueness of $\mathbb{P}$.
\end{proof}
\begin{remark}
\noindent
The set $\Sigma_E$ is the intersection of the sets $S_{J,K}(E) \in 
\mathcal{S}(E)$, indexed 
by the pairs $(J,K) \in \mathcal{F}(E) \times \mathcal{F}(E)$ such that $J \subset K$, 
and defined as follows: a family $(\sigma_I)_{I \in \mathcal{F}(E)}$ is in 
$S_{J,K}(E)$ if and only if $\sigma_J = \pi_{K,J} (\sigma_K)$. If $E$ is countable, 
this intersection is countable, and then $\Sigma_E \in \mathcal{S}(E)$. However, if $E$ is 
uncountable, $\Sigma_E$ does not seem to be measurable, and we do not know how to construct 
the measure $\mathbb{P}$ in this case. 
\end{remark}
\noindent
An immediate consequence of Proposition \ref{ewensvi} is the following:
\begin{corollary} \label{oquy}
Let $E$ be a countable set, and let  $\mathcal{S}_E$ be the $\sigma$-algebra on $\Sigma_E$,
 generated by the coordinates mappings $(\sigma_I)_{I \in \mathcal{F}(E)}
\mapsto \sigma_J$ for $J \in \mathcal{F}(E)$. 
Let $(\mathbb{P}_I)_{I \in \mathcal{F}(E)}$ be a family of probability measures, $\mathbb{P}_I$ defined on
$\Sigma_I$, which is compatible in the sense of Proposition \ref{ewensvi}. Then, there exists a unique probability measure 
$\mathbb{P}_E$ on the measurable space $(\Sigma_E, \mathcal{S}_E)$ such that for all $I \in \mathcal{F}(E)$, the image of $\mathbb{P}_E$ 
by the coordinate indexed by $I$ is equal to $\mathbb{P}_I$. 
\end{corollary}
\noindent
An example of measure on $\Sigma_E$ which can be constructed by using Proposition \ref{ewensvi} and Corollary \ref{oquy} is the
Ewens measure, which, for $E = \mathbb{N}^*$, is studied in detail in \cite{Tsi}. The precise existence result in our more general setting is the following:
\begin{proposition}
Let $\theta$ be in $\mathbb{R}_+$ and $E$ a countable set. For all $I \in \mathcal{F}(E)$, let $\mathbb{P}_I^{(\theta)}$ be 
the Ewens measure of parameter $\theta$, defined as follows: for all $\sigma \in \Sigma_I$, 
$$\mathbb{P}_I^{(\theta)} (\sigma) = \frac{ \theta^{n-1}}{ (\theta + 1) \dots (\theta + N-1) },$$
where $n$ is the number of cycles of $\sigma$ and $N$ the cardinal of $I$. Then the family $(\mathbb{P}_I^{(\theta)})_{I \in \mathcal{F}(E)}$ is 
compatible in the sense of Proposition \ref{ewensvi}: the measure $\mathbb{P}_E^{(\theta)}$ on $(\Sigma_E, \mathcal{S}_E)$ which 
is deduced from Corollary \ref{oquy} is called the {\rm Ewens measure of parameter $\theta$}.
\end{proposition}
\noindent
The Ewens measures are particular cases of the so-called {\it central measures}. Indeed, for all
 countable sets $E$ and for all $g \in \Sigma_E^{(0)}$, the conjugation by $g$ defined in 
Proposition \ref{conj} is measurable with respect to the $\sigma$-algebra $\mathcal{S}_E$. Therefore it defines a group action on 
the set of probability measures on $(\Sigma_E, \mathcal{S}_E)$: by definition, a probability measure $\mathbb{P}$ is central if and only if it is invariant by 
this action. It is easy to check that this condition holds if and only if all the image measures of $\mathbb{P}$ by the coordinate maps are invariant by 
conjugation (which is clearly the case for the Ewens measures). In the case of $E = \mathbb{N}^*$, Tsilevich  (in \cite{Tsi98}) has completely characterized 
the central measures, by using the properties of the partitions of countable sets, described by Kingman (see \cite{King75}, \cite{King78a}, \cite{King78b}). 
If the cardinality of a finite set $I$ is denoted by $|I|$, then the result by Tsilevich can  easily be translated in our framework as follows:
\begin{proposition} \label{loup37}
Let $E$ be a countable set, $\mathbb{P}$ a central measure on $(\Sigma_E, \mathcal{S}_E)$, and $\sigma$ a random virtual permutation following
the distribution $\mathbb{P}$. Then,
for all $x, y \in E$, the indicator of the event $\{x \sim_{\sigma} y \}$ is measurable with respect to the $\sigma$-algebra $\mathcal{S}_E$,  and 
 there exists a family of random variables $(\lambda(x))_{x \in E}$, unique up to almost sure equality, taking their values in the interval $[0,1]$, and satisfying the following 
properties:
\begin{itemize}
\item  For $I \in \mathcal{F}(E)$,
 $$ \frac{|I \cap \mathcal{C}_{\sigma} (x)|}{|I|} \underset{|I| \rightarrow
 \infty}{\longrightarrow}
 \lambda(x),$$
in $L^1$, and then in all the spaces $L^p$ for $p \in [1, \infty)$ since all the variables involved here
are bounded by one.
\item For any strictly increasing sequence $(I_n)_{n \geq 1}$ of nonempty sets
 in $\mathcal{F}(E)$, 
$$ \frac{|I_n \cap \mathcal{C}_{\sigma} (x)|}{|I_n|} \underset{n \rightarrow \infty}{\longrightarrow}
 \lambda(x)$$
almost surely.
\item Almost surely, for all $x \in E$, $\lambda(x) = 0$ if and only if $x$ is a fixed point of $\sigma$, i.e. a fixed point of $\sigma_I$ for all $I \in 
\mathcal{F}(E)$ containing $x$.
\item Almost surely, for all $x, y \in E$, $\lambda(x) = \lambda (y)$ if $x \sim_{\sigma} y$.
\end{itemize}
\noindent
For all $k \geq 1$, let us then define $\lambda_k$ as the 
supremum of $\inf_{1 \leq j \leq k} \lambda(x_j)$ for a sequence $(x_j)_{1 \leq j \leq k}$ of elements in $E$ satisfying 
$x_j \not\sim_{\sigma} x_{j'}$ for $j \neq j'$, the supremum of the empty set being taken to be equal to zero. The sequence $(\lambda_k)_{k \geq 1}$
is a random variable with values in the space of the sequences of elements in $[0,1]$, endowed with the 
$\sigma$-algebra generated by the coordinate maps, $(\lambda_k)_{k \geq 1}$ is uniquely determined up to almost sure equality, and 
it lies almost surely in the (measurable) simplex $\Lambda$ of non-increasing random sequences in $[0,1]$, satisfying 
the inequality:
$$\sum_{k \geq 1} \lambda_k \leq 1.$$
Moreover, the central measure $\mathbb{P}$ is uniquely determined by the distribution of the sequence $(\lambda_k)_{k \geq 1}$, 
as follows:
\begin{itemize}
\item For any probability measure $\nu$ on the space $\Lambda$ endowed with the $\sigma$-algebra of the coordinate maps, 
  there exists a unique central measure $\mathbb{P}_{\nu}$ on $(\Sigma_E, \mathcal{S}_E)$
such that the corresponding random sequence $(\lambda_k)_{k \geq 1}$ has distribution $\nu$.
\item For all sequences $\lambda \in \Lambda$, let $\mathbb{P}_{\lambda} := \mathbb{P}_{\delta_{\lambda}}$, where
$\delta_{\lambda}$ denotes the Dirac measure at $\lambda$. Then, for all $A \in \mathcal{S}(E)$, the 
map $\lambda \mapsto \mathbb{P}_{\lambda}(A)$ is measurable and for all probability measures $\nu$, one 
has:
$$\mathbb{P}_{\nu} (A) = \int_{\Lambda} \mathbb{P}_{\lambda}(A) d \nu(\lambda).$$
\end{itemize}
 \end{proposition}
 
 \begin{remark}
 Intuitively, $\lambda(x)$ represents the asymptotic length of the cycle of $x$ and $(\lambda_k)_{k \geq 1}$ is the non-increasing 
 sequence of cycle lengths. For all $\theta \geq 0$, the Ewens measure of parameter $\theta$ is equal to $\mathbb{P}_{\nu^{(\theta)}}$,  where 
 $\nu^{(\theta)}$ is the Poisson-Dirichlet distribution of parameter $\theta$ (Dirac measure at the sequence $(1,0,0,\dots)$ if $\theta = 0$). 
  \end{remark}
\begin{proof}
The measurability of the equivalence relation $\sim_{\sigma}$ is an immediate consequence of the fact that the event $x \sim_{\sigma} y$ 
depends only on the permutation $\sigma_{\{x,y\}}$. Therefore, $\pi_0 : \sigma \mapsto \, \sim_{\sigma}$ is a measurable map from 
$(\Sigma_E, \mathcal{S}_E)$ to the space of equivalence relations of $E$, endowed with the $\sigma$-algebra generated 
by the events of the form $\{x \sim y\}$ for $x, y \in E$. Since this measurable space is canonically identified with a measurable 
space $(\Pi_E, \mathcal{V}_E)$ such that $\Pi_E$ is the space of partitions of the set $E$, the map $\pi_0$ can be identified to
a measurable map $\pi$ from $(\Sigma_E, \mathcal{S}_E)$ to $(\Pi_E, \mathcal{V}_E)$. This map induces a map from the 
probability measures on $(\Sigma_E, \mathcal{S}_E)$ to the probability measures on $(\Pi_E, \mathcal{V}_E)$, and 
Proposition \ref{conj} implies that a central measure is always mapped to the distribution of an exchangeable partition. 
Let us now prove that this correspondence is bijective, i.e. for any distribution $\mathbb{Q}$ on $(\Pi_E, \mathcal{V}_E)$ inducing an 
exchangeable partition, there exists a unique central measure $\mathbb{P}$ on $(\Sigma_E, \mathcal{S}_E)$ such that 
the image of $\mathbb{P}$ by $\pi$ is equal to $\mathbb{Q}$. Indeed, this condition is satisfied if and only if for all $J \in \mathcal{F}(E)$, 
the law of the partition of $J$ induced by the cycle structure of $\sigma_J$, where $(\sigma_I)_{I \in \mathcal{F}(E)}$ follows 
the distribution $\mathbb{P}$, is equal to 
the law of the restriction to $J$ of a partition following the distribution $\mathbb{Q}$. Since a probability on $\Sigma_J$ which is 
invariant by conjugation is completely determined by the corresponding distribution of the cycle lengths, the uniqueness of $\mathbb{P}$ 
follows. Now, for all $J \in \mathcal{F}(E)$, and all distributions $\mathbb{Q}_J$ on the partitions of $J$, let us define $\pi^{-1}(\mathbb{Q}_J)$
as the distribution of a random permutation $\sigma_J \in \Sigma_J$ satisfying the following conditions:
\begin{itemize}
\item The partition of $J$ induced by the cycle structure of $\sigma_J$ follows the distribution $\mathbb{Q}_J$.
\item Conditionally on this partition, the law of $\sigma_J$ is uniform.
\end{itemize}
\noindent
The existence of a measure $\mathbb{P}$ satisfying the conditions above is then a consequence of the following property of compatibility, which 
can be easily checked: for all $I, J \in \mathcal{F}$ such that $I \subset J$, for all distributions $\mathbb{Q}_J$ on the partitions of $J$, 
the image of $\pi^{-1}(\mathbb{Q}_J)$ by the map from $\Sigma_J$ to $\Sigma_I$ which removes the elements of $J \backslash I$ from the 
cycle structure is equal to $\pi^{-1}(\mathbb{Q}_I)$, where $\mathbb{Q}_I$ is the image of $\mathbb{Q}_J$ by the restriction of the partitions to the 
set $I$. Now, the bijective correspondence induced by $\pi$ implies that it is sufficient to show the equivalent of Proposition \ref{loup37} 
for exchangeable partitions. Let $(x_n)_{n \geq 1}$ be an enumeration of the set $E$, and let us define the strictly increasing 
family of sets $(I^{(0)}_n)_{n \geq 1}$, by:
$$ I^{(0)}_n := \{x_1, x_2, \dots, x_n\}.$$
If in Proposition \ref{loup37} for exchangeable partitions, the convergence of $|I \cap \mathcal{C}(x)|/ |I|$ to $\lambda(x)$ in $L^1$ when
$|I|$ goes to infinity is removed, and if the almost sure convergence along any strictly increasing sequence of nonempty sets in $\mathcal{F}(E)$ 
is replaced by the convergence only along the sequence $(I^{(0)}_n)_{n \geq 1}$, then the remaining of the proposition is a direct consequence of classical  
results by Kingman. Let us now prove the convergence in $L^1$, fixing $x \in E$. 
By dominated convergence, 
\begin{equation}
 \frac{|\mathcal{C}(x) \cap I^{(0)}_{n}|}{|I^{(0)}_{n}|} \underset{n \rightarrow \infty}{\longrightarrow} \lambda(x), \label{y600}
 \end{equation}
in $L^1$ when $n$ goes to infinity. Moreover, for all $J, K \in \mathcal{F}(E)$ such that  $\{x\} \subset J \subset K$, 
the joint law of $|\mathcal{C}(x) \cap J|$ and $|\mathcal{C}(x) \cap K|$ depends only on $|J|$ and 
$|K|$, since the underlying probability measure on $(\Sigma_E, \mathcal{S}_E)$ is central. 
Now, for $|J|$ large enough, $I^{(0)}_{|J|}$, and a fortiori $I^{(0)}_{|K|}$, contain $x$. In this case, 
one has 
\begin{equation}
\mathbb{E} \left[ \left| \frac{|\mathcal{C}(x) \cap J|}{|J|} - 
\frac{|\mathcal{C}(x) \cap K|}{|K|} \right|\right] = 
\mathbb{E} \left[ \left| \frac{|\mathcal{C}(x) \cap I^{(0)}_{|J|}|}{|I^{(0)}_{|J|}|} - 
 \frac{|\mathcal{C}(x) \cap I^{(0)}_{|K|}|}{|I^{(0)}_{|K|}|} \right| \right]. \label{nnnnn}
\end{equation} 
By the convergence \eqref{y600}, the right-hand side of \ref{nnnnn} tends to zero when $|J|$ goes to 
infinity. In other words, there exists a sequence $(\epsilon(n))_{n \geq 1}$, decreasing to zero at infinity, 
such that
$$\mathbb{E} \left[ \left| \frac{|\mathcal{C}(x) \cap J|}{|J|} - 
\frac{|\mathcal{C}(x) \cap K|}{|K|} \right|\right] \leq \epsilon(|J|),$$
for all $J, K \in \mathcal{F}(E)$ such that $\{x\} \subset J \subset K$. 
 Now, for all $J, K \in \mathcal{F}(E)$ 
containing $x$, we obtain, by taking $J \cup K$ as an intermediate set and by using the triangle inequality:
\begin{equation}
\mathbb{E} \left[ \left| \frac{|\mathcal{C}(x) \cap J|}{|J|} - 
\frac{|\mathcal{C}(x) \cap K|}{|K|} \right|\right] \leq \epsilon(|J|) + \epsilon(|K|). \label{tam}
\end{equation}
Moreover, if $J \in \mathcal{F}(E)$ does not contain $x$, we have:
\begin{equation}
 \frac{|\mathcal{C}(x) \cap (J \cup \{x\})|}{|J \cup \{x\}|}  - \frac{1}{|J|} \leq
 \frac{|\mathcal{C}(x) \cap J|}{|J|} \leq  \frac{|\mathcal{C}(x) \cap (J \cup \{x\})|}{|J \cup \{x\}|}. 
\label{nghyik}
\end{equation}
\noindent
Combining \eqref{tam} and \eqref{nghyik}, we deduce, for any $J, K \in \mathcal{F}(E)$:
$$\mathbb{E} \left[ \left| \frac{|\mathcal{C}(x) \cap J|}{|J|} - 
\frac{|\mathcal{C}(x) \cap K|}{|K|} \right|\right] \leq \epsilon(|J|) + \epsilon(|K|) + \frac{1}{|J|}
+ \frac{1}{|K|}.$$
In particular, by taking $K = I^{(0)}_n$, one has for $n \geq |J|$:
$$\mathbb{E} \left[ \left| \frac{|\mathcal{C}(x) \cap J|}{|J|} - 
\frac{|\mathcal{C}(x) \cap I^{(0)}_n|}{|I^{(0)}_n|} \right|\right] 
\leq 2 \epsilon(|J|) + \frac{2}{|J|},$$
and, after letting $n \rightarrow \infty$:
$$\mathbb{E} \left[ \left| \frac{|\mathcal{C}(x) \cap J|}{|J|} - \lambda (x)
\right| \right] \leq  2 \epsilon(|J|) + \frac{2}{|J|},$$
which proves the convergence in $L^1$.  
Now, let $(I_n)_{n \geq 1}$ be a strictly increasing sequence of sets in $\mathcal{F}(E)$, containing $x$ and such that $|I_n| = n$. 
Since the underlying measure on $\Sigma_E$ is central, the law of the sequence $(|\mathcal{C}(x) \cap I_n|/n)_{n \geq 1}$ 
is independent of the choice of $(I_n)_{n \geq 1}$. Hence, $(|\mathcal{C}(x) \cap I_n|/n)_{n \geq 1}$ is almost surely a Cauchy sequence, 
since it is the case when we suppose that $I_n = I^{(0)}_n$ for $n$ large enough. By the convergence in $L^1$ proven above, 
the limit of $(|\mathcal{C}(x) \cap I_n|/n)_{n \geq 1}$ is necessarily $\lambda(x)$ almost surely. The almost sure convergence is then 
proven for any strictly increasing sequence $(I_n)_{n \geq 1}$ such that $x \in I_n$ and $|I_n| = n$
for all $n \geq 1$: the first condition can be removed simply by putting $x$ into all the sets $(I_n)_{n \geq 1}$, which changes the quotient
$|\mathcal{C}(x) \cap I_n|/|I_n|$ by at most $1/|I_n|$, the second one can be suppressed by inserting any strictly increasing sequence of 
nonempty sets in $\mathcal{F}(E)$ into an increasing sequence $(I_n)_{n \geq 1}$ such that $|I_n| = n$ for all $n \geq 1$.
\end{proof}
\noindent
As we have seen in its proof, the result of Proposition \ref{loup37} gives some information about the asymptotic cycle lengths of a random virtual permutation 
following a central measure, but does not tell anything  about the relative positions of the elements inside their cycle. It is a remarkable fact  that these relative positions  
also admit an asymptotic limit when the size of the sets in $\mathcal{F}(E)$ goes to infinity:
\begin{proposition} \label{qwerty}
Let $E$ be a countable set, $x, y$ two elements of $E$ and  $\sigma$ a virtual permutation of $E$, which follows a central probability measure.
 Then, on the event $\{x \sim_{\sigma} y\}$: 
\begin{itemize}
\item For all $I \in \mathcal{F}(E)$ containing $x$ and $y$, there exists a unique integer $k_I(x,y) \in \{0,1,\dots,| I \cap \mathcal{C}_{\sigma}(x) | -1 \}$
such that $\sigma_I^{k} (x) = y$;
\item The variable $k_I(x,y)/|I|$ converges in $L^1$, and then in all the spaces $L^p$ for $p \in [1, \infty)$, to a limit random variable $\Delta(x,y) \in [0,1]$ when 
$|I|$ goes to infinity
\item For any strictly increasing sequence $(I_n)_{n \geq 1}$ of sets in $\mathcal{F}(E)$ containing $x$ and $y$,  $k_{I_n}(x,y)/|I_n|$ converges 
almost surely to $\Delta(x,y)$ when $n$ goes to infinity.
\end{itemize}
\end{proposition}
\begin{proof}
Since Proposition \ref{qwerty} is obvious for $x=y$ (with $\Delta(x,x) = 0$), let us suppose that $x \neq y$. If $x \sim y$, and if $I \in \mathcal{F}(E)$ 
contains $x$ and $y$, then these two elements of $E$ are in the same cycle of $\sigma_I$, which has length
$| I \cap \mathcal{C}_{\sigma}(x) | $. This implies the existence and the uniqueness of $k_I(x,y)$. Now, let 
$(I_n)_{n \geq 1}$ be a strictly increasing sequence of sets in $\mathcal{F}(E)$ such that $I_1 = \{x,y\}$, and for all $n \geq 1$, $I_{n+1} = I_n \cup \{z_n\}$ with
$z_n \notin I_n$. Then, for all $n \geq 1$, conditionally on $\sigma_{I_n}$, and on the event $x \sim y \sim z_n$, the position of $z_n$ inside the 
cycle structure of $\sigma_{I_{n+1}}$ is uniform among all the possible positions in the cycle containing $x$ and $y$, since the  $\sigma$ follows a central probability
measure. Hence, again conditionally on $\sigma_{I_n}$, $x \sim y \sim z_n$, one has $k_{I_{n+1}}(x,y) = k_{I_n}(x,y) +1$ with probability 
$k_{I_n}(x,y)/  |\mathcal{C}(x) \cap I_{n}| $ (if $z_n$ is inserted between $x$ and $y$ in their common cycle), and 
 $k_{I_{n+1}}(x,y) = k_{I_n}(x,y)$ with probability $1 - (k_{I_n}(x,y)/  |\mathcal{C}(x) \cap I_{n}|)$.
One deduces that $(k_{I_n}(x,y) \mathds{1}_{x \sim y}/ |\mathcal{C}(x) \cap I_{n}|)_{n \geq 1}$ is a martingale with respect to the filtration generated by 
$(\sigma_{I_n})_{n \geq 1}$. Since it takes its values in $[0,1]$, it converges almost surely to a limit random variable. Moreover, on the event 
$\{x \sim y\}$, the quotient $|\mathcal{C}(x) \cap I_{n}| / |I_n|$ 
converges almost surely to $\lambda(x)$, which is almost surely strictly positive, since $x$ is not a fixed point of $\sigma$. Hence, again on the set $x \sim y$, 
 $k_{I_n}(x,y)/|I_n|$ converges 
to a random variable $\Delta(x,y)$ almost surely, and then in $L^1$ by dominated convergence. 
Now, let $J, K$ be two finite subsets of $E$ containing $x$ and $y$. If $J$ is included in $K$, then by the $L^1$ convergence of  $k_{I_n}(x,y)/|I_n|$ at infinity, and 
by the centrality of the law of $\sigma$, one has:
\begin{align*}
\mathbb{E} [|(k_J(x,y)/|J|) - (k_K(x,y)/|K|)| \mathds{1}_{x \sim y} ] & = \mathbb{E} [|(k_{I_{|J|-1}} (x,y)/|J|) -(k_{I_{|K|-1}} (x,y)/|K|) |
 \mathds{1}_{x \sim y} ] \\ &  \leq \epsilon(|J|),
 \end{align*}
where $\epsilon(n)$ decreases to zero at infinity. If $J$ is not supposed to be included in $K$, then by using $J \cup K$ as an intermediate step, one obtains:
$$ \mathbb{E} [|(k_J(x,y)/|J|) - (k_K(x,y)/|K|)| \, \mathds{1}_{x \sim y}] \leq \epsilon(|J|) + \epsilon(|K|).$$
In particular, for $J \in \mathcal{F}(E)$ and for all $n \geq |J|$:
$$  \mathbb{E} [|(k_J(x,y)/|J|) - (k_{I_n}(x,y)/|I_n|)| \, \mathds{1}_{x \sim y} ] \leq 2 \epsilon(|J|),$$
which implies that 
$$  \mathbb{E} [|(k_J(x,y)/|J|) - \, \Delta(x,y)| \mathds{1}_{x \sim y} ] \leq 2 \epsilon(|J|).$$
Therefore, on the set $x \sim y$, $k_J(x,y)/|J|$ converges in $L^1$ to  $\Delta(x,y)$. Now, let $(J_n)_{n \geq 1}$ be a strictly increasing sequence of sets in 
$\mathcal{F}(E)$, containing $x$ and $y$. This sequence can be inserted in a sequence with satisfies the same assumptions as $(I_n)_{n \geq 1}$, 
which implies that on the event $\{x \sim y\}$,  $k_{J_n}(x,y)/|J_n|$ converges almost surely to a random variable, necessarily equal 
to  $\Delta(x,y)$ since $k_{J_n}(x,y)/|J_n|$ converges to $\Delta(x,y)$ in $L^1$.
\end{proof}
\noindent
For $x \sim y$, $\Delta(x,y)$ represents the asymptotic number of iterations of the virtual permutation $\sigma$ which are needed to go from $x$ to $y$, 
and $\Delta(x,y)/\lambda(x)$ is the asymptotic proportion of the cycle $\mathcal{C}(x)$, lying between $x$ and $y$. The distribution of 
this proportion is deduced from the following result:
\begin{proposition} \label{uniform}
Let $E$ be a countable set, and  let $\sigma$ be a virtual permutation of $E$, which follows a central probability measure. For $m, n_1, \dots n_m \geq 0$, let 
$(x_j)_{1 \leq j \leq m}$ and $(y_{j,k})_{1 \leq j \leq m, 1 \leq k \leq n_j}$ be distinct elements of $E$. 
Then, conditionally of the event: $$\{ \forall j \in \{1, \dots, m\} , \forall k \in \{1, \dots, n_j\}, x_j \sim_{\sigma} y_{j,k} \} \cap \{ \forall j_1 \neq j_2 \in \{1, \dots, m\}, 
x_{j_1} \not\sim x_{j_2} \},$$
the variables $(\Delta(x_j, y_{j,k})/\lambda(x_j))_{1 \leq j \leq m, 1 \leq k \leq n_j}$ are almost surely well-defined, independent and uniform on $[0,1]$, and they
 form a family which is independent of the variables $(\lambda(x))_{x \in E}$.
\end{proposition}
\begin{remark}
It is possible to take some indexes $j$ such that $n_j = 0$. In this case, there is no point of the form $y_{j,k}$, and then no variable
of the form $\Delta(x_j, y_{j,k})/\lambda(x_j)$, involved in Proposition \ref{uniform}.
\end{remark} 
\begin{proof}
Let $q \geq 1$, and let $(x_j)_{m+1 \leq j \leq m+q}$ be a family of elements of $E$ such that all the elements $(x_j)_{1 \leq j \leq m+q}$ are distinct. 
 Let us denote:
$$\mathcal{E} :=  \{ \forall j \in \{1, \dots, m\} , \forall k \in \{1, \dots, n_j\}, x_j \sim_{\sigma} y_{j,k} \} \cap \{ \forall j_1 \neq j_2 \in \{1, \dots, m\}, 
x_{j_1} \not\sim x_{j_2} \}.$$
Moreover, for all $I \in \mathcal{F}(E)$ containing $x_j$ for all $j \in \{1, \dots, m+q\}$ and $y_{j,k}$ for all $j \in \{1, \dots, m\}$, $k \in \{1, \dots n_j\}$, and 
for all sequences $(p_j)_{1 \leq j \leq m+q}$ of strictly positive integers such that $p_j > n_j$ for all $j \in \{1, \dots, m\}$, let us define the following event:
$$\mathcal{E}_{I, (p_j)_{1 \leq j \leq m}} := \mathcal{E} \cap \{ \forall j \in \{1, \dots m+q\}, |\mathcal{C}(x_j) \cap  I | = p_j \}.$$
By the centrality of the law of $\sigma$, conditionally on $\mathcal{E}_{I, (p_j)_{1 \leq j \leq m+q}}$, the sequences $(k_{I} (x_j, y_{j,k}))_{1 \leq k \leq n_j}$, for 
$j \in \{1, \dots, m\}$ such that $n_j > 0$, are independent, and the law of  $(k_{I} (x_j, y_{j,k}))_{1 \leq k \leq n_j}$ is uniform among all the possible sequences of $n_j$ distinct elements
in $\{1, 2, \dots, p_j - 1\}$. Now, for all $j \in \{1, \dots, m\}$, let $\Phi_j$ be a continuous and bounded function from $\mathbb{R}^{n_j}$ to $\mathbb{R}$ ($\Phi_j$ is a
real constant if $n_j = 0$). Since for $n_j > 0$, the uniform distribution
on the sequences of $n_j$ district elements in $\{1/p, 2/p, \dots, (p-1)/p\}$ converges weakly to the uniform distribution on $[0,1]^{n_j}$ when $p$ goes to infinity, one has:
\begin{align*}
 & \; \mathbb{E} \left[ \prod_{j=1}^m \Phi_j  \left( (k_{I} (x_j, y_{j,k})/ p_j )_{1 \leq k \leq n_j} \right) | \, \mathcal{E}_{I, (p_j)_{1 \leq j \leq m+q}} \right] 
\\ & =  \prod_{j=1}^m \mathbb{E} \left[  \Phi_j  \left( (k_{I} (x_j, y_{j,k})/ p_j )_{1 \leq k \leq n_j} \right)  | \, \mathcal{E}_{I, (p_j)_{1 \leq j \leq m+q}} \right] \\ 
& = \left(\prod_{j=1}^m \int_{[0,1]^{n_j}} \Phi_j \right) + \alpha(  (\Phi_j, p_j)_{1 \leq j \leq m} ),
\end{align*}
where for fixed functions $(\Phi_j)_{1 \leq j \leq m}$, $|\alpha( (\Phi_j, p_j)_{1 \leq j \leq m} )|$ is uniformly bounded and tends to zero when the
 minimum of $p_j$ for $n_j > 0$ goes to infinity. One deduces that for all continuous, bounded functions $\Psi$ from $\mathbb{R}^m$ to $\mathbb{R}$:
$$
 \mathbb{E} \left[ \Psi \left( \left(  \frac{|\mathcal{C}(x_j) \cap  I |}{|I|}  \right)_{1 \leq j \leq m+q} \right)
 \prod_{j=1}^m \Phi_j  \left( (k_{I} (x_j, y_{j,k})/ (|\mathcal{C}(x_j) \cap  I |) )_{1 \leq k \leq n_j} \right) | \mathcal{E} \right] $$
 \begin{align*}
 & =  \mathbb{E} \left[ \Psi \left(\left. \left(  \frac{|\mathcal{C}(x_j) \cap  I |}{|I|}  \right)_{1 \leq j \leq m+q} \; \right| \mathcal{E} \right) \right] \, \left(\prod_{j=1}^m \int_{[0,1]^{n_j}} \Phi_j \right) 
\\ & +  \mathbb{E} \left[ \Psi \left( \left. \left( \frac{|\mathcal{C}(x_j) \cap  I |}{|I|}  \right)_{1 \leq j \leq m+q} \right) \alpha( (\Phi_j, |\mathcal{C}(x_j) \cap  I|)_{1 \leq j \leq m} ) \; 
\right| \mathcal{E} \right].
\end{align*}
Now, for any strictly increasing sequence $(I_r)_{r \geq 1}$ of sets satisfying the same assumptions as $I$, and on the event $\mathcal{E}$: 
\begin{itemize}
\item For all $j \in \{1, \dots, m+q\}$, $|\mathcal{C}(x_j) \cap  I_r |/|I_r|$ converges almost surely to $\lambda(x_j)$ when $r$ goes to infinity.
\item For all $j \in \{1, \dots, m\}$ and $k \in \{1, \dots, n_j\}$, $k_{I_r} (x_j, y_{j,k})/ (|\mathcal{C}(x_j) \cap  I_r |)$ tends almost surely
 to $\Delta(x_j, y_{j,k})/(\lambda(x_j))$, which is well-defined since $\lambda(x_j) > 0$ almost surely (by assumption, $x_j$ is not a fixed point of $\sigma$ if $n_j > 0$). 
\item For all $j \in \{1, \dots, m\}$ such that $n_j > 0$, and then $\lambda(x_j) > 0$ almost surely, $|\mathcal{C}(x_j) \cap  I_r |$ tends almost surely to infinity, which implies that 
$\alpha( (\Phi_j, |\mathcal{C}(x_j) \cap  I|)_{1 \leq j \leq m} ) $ goes to zero.
\end{itemize}
\noindent
By dominated convergence, one deduces:
$$ \mathbb{E} \left[ \Psi \left( \left( \lambda(x_j) \right)_{1 \leq j \leq m+q} \right)
 \prod_{j=1}^m \Phi_j  \left(\Delta(x_j, y_{j,k})/(\lambda(x_j) ) )_{1 \leq k \leq n_j} \right) | \mathcal{E} \right] $$
 $$ =  \mathbb{E} \left[ \Psi \left( \left( \lambda(x_j) \right)_{1 \leq j \leq m+q} \right) | \mathcal{E} \right] \, \left(\prod_{j=1}^m \int_{[0,1]^{n_j}} \Phi_j \right).$$
 This proves Proposition \ref{uniform} with $(\lambda(x))_{x \in E}$ replaced by $\lambda(x_j)_{j \in \{1, \dots, m+q\}}$. By increasing $q$ and by using 
 monotone class theorem, we are done. 
\end{proof}
\noindent
A particular case of Proposition \ref{uniform} is the following:
\begin{corollary}
Let $E$ be a countable set, $x$ and $y$ two distinct elements of $E$, and $\sigma$ a virtual permutation of $E$, which follows a central 
probability measure. Then, conditionally on the event $\{x \sim_{\sigma} y\}$, the variable $\Delta(x,y)/\lambda(x)$ is almost surely 
well-defined and uniform on the interval $[0,1]$ (which implies 
that $\Delta(x,y) \in (0, \lambda(x))$ almost surely). 
\end{corollary}
\noindent
Recall now that by definition, for $x \sim y$, $I \in \mathcal{F}(E)$ containing $x$ and $y$, $k_I (x,y)$ is the unique integer between $0$ and $|\mathcal{C}(x) \cap I | -1$
such that $\sigma^{k_I (x,y)}(x) = y$. If the condition $0 \leq k_I (x,y) \leq |\mathcal{C}(x) \cap I | -1$ is removed, then $k_I(x,y)$ is only defined 
up to a multiple of $|\mathcal{C}(x) \cap I |$. Hence, by taking $|I| \rightarrow \infty$, it is natural to introduce the class $\delta(x,y)$ of $\Delta(x,y)$ modulo 
$\lambda(x)$. This class satisfies the following properties:
\begin{proposition} \label{delta}
Let $E$ be a countable set, $x, y, z$ three elements of $E$ (not necessarily distinct), and $\sigma$ a virtual permutation of $E$, which follows a central probability 
measure. Then, almost surely on the event $\{x \sim_{\sigma} y \sim_{\sigma} z\}$:
\begin{itemize}
\item $\delta(x,x) = 0$;
\item $\delta(x,y) = - \delta(y,x)$;
\item $\delta(x,y) + \delta(y,z) = \delta(x,z)$.
\end{itemize}
\end{proposition}
\begin{proof}
If $x = y =z$ and if $x$ is a fixed point of $\sigma$, then Proposition \ref{delta} is trivial (with all the values of $\delta$ equal to zero). Hence, we can suppose 
that $x$ is not a fixed point of $\sigma$. 
Let us now prove the third item, which implies immediately the first (by taking $x = y = z$) and then the second (by taking $x=z$). For any $I \in \mathcal{F}(E)$ containing
$x$, $y$ and $z$, one has necessarily, on the event $\{x \sim y \sim z\}$:
 $$\sigma_I^{k_I(x,y) + k_I(y,z)}(x) =  \sigma_I^{k_I(x,z)}(x) = z,$$
 which implies 
 $$k_I(x,y) + k_I(y,z) - k_I(x,z) \in \{0, |\mathcal{C}(x) \cap I | \}$$
 since $k_I(x,y)$,  $k_I(y,z)$ and  $k_I(x,z)$ are in the set $\{0,1, \dots,  |\mathcal{C}(x) \cap I | - 1\}$. Let $(I_n)_{n \geq 1}$ be a strictly increasing 
 sequence of sets in $\mathcal{F}(E)$ containing $x$, $y$ and $z$: the sequence $$\left( \frac{k_{I_n}(x,y) + k_{I_n}(y,z) - 
 k_{I_n}(x,z)}{|\mathcal{C} (x)  \cap I_n|} \right)_{n \geq 1}$$
 of elements in $\{0,1\}$ converges almost surely to $(\Delta(x,y) + \Delta(y,z) - \Delta(x,z))/(\lambda(x))$ (recall that $\lambda(x) > 0$ since 
 $x$ is not a fixed point of $\sigma$). One deduces that $\Delta(x,y) + \Delta(y,z) - \Delta(x,z)$
 is almost surely equal to $0$ or $\lambda(x)$, which implies that  $\delta(x,y) + \delta(y,z) = \delta(x,z)$. 
\end{proof}
\noindent
The properties of $\delta$ suggest the following representation of the cycle structure of $\sigma$: for $x \in E$ which is not a fixed point of $\sigma$, one puts 
the elements of the cycle of $x$ 
for $\sim_{\sigma}$ into a circle of perimeter $\lambda(x)$, in a way such that for two elements $y, z \in \mathcal{C}(x)$, $\delta(y,z)$ is the length, counted 
counterclockwise, of the circle arc between $x$ and $y$. In order to make the representation precise, we shall fix the asymptotic cycle
lengths of $\sigma$, by only considering the central measures of the form $\mathbb{P}_{\lambda}$, for a given sequence $\lambda$ in the simplex 
$\Lambda$. When we make this assumption, we do not lose generality, since by Proposition \ref{loup37}, any central measure on $\Sigma_E$ can be written as a mixture of the measures of the form $\mathbb{P}_{\lambda}$ (i.e. an integral with respect to a probability measure $\nu$ on $\Lambda$), and we obtain the following result:
\begin{proposition} \label{circles}
Let $\lambda := (\lambda_k)_{k \geq 1}$ be a sequence in the simplex $\Lambda$, let $E$ be a countable set, and let $\sigma = (\sigma_I)_{I \in \mathcal{F}(E)}$ be
 a virtual permutation of $E$, following the central probability measure $\mathbb{P}_{\lambda}$ defined in Proposition \ref{loup37}. 
For $k \geq 1$ such that $\lambda_k > 0$, let $C_k$ be a circle of perimeter $\lambda_k$, the circles being pairwise disjoint, and 
for all $x, y \in C_k$, let $y-x \in \mathbb{R}/\lambda_k 
\mathbb{Z}$ be the length of the arc of circle from $x$ to $y$, counted counterclockwise and modulo $\lambda_k$. 
Let $L$ be a segment of length $1 - \sum_{k \geq 1} \lambda_k$ (the empty set if $\sum_{k \geq 1} \lambda_k = 1$), disjoint of the circles $C_k$, and let $C$
be the unions of all the circles $C_k$ and $L$. 
 Let $\mu$ the uniform probability measure on $C$ (endowed with the 
$\sigma$-algebra generated by the Borel sets of $L$ and $C_k$, $k \geq 1$), i.e. the unique measure 
such that $\mu(A)$ is equal to the length of $A$, for any set $A$ equal to an arc of one of the circles $C_k$, or a segment included in $L$.
Then, after enlarging, when necessary, the probability space on which $\sigma$ is defined,  
 there exists a family $(X_x)_{x \in E}$ of i.i.d. random variables on the space $C$, following the law $\mu$, and such that 
 the following conditions hold almost surely:
 \begin{itemize}
\item For all $x \in E$, $X_x \in L$ if and only if $x$ is a fixed point of $\sigma$.
\item For all $x$, $y$, distinct elements of $E$, $x \sim_{\sigma} y$ if and only if $X_x$ and $X_y$ are on the same circle $C_k$, and 
 in this case, $\lambda(x) = \lambda(y) = \lambda_k$.
\item For all $x$, $y$, distinct elements of $E$ such that $x \sim_{\sigma} y$, $\delta(x,y)$ is equal to $X_y-X_x$ modulo $\lambda(x)$. 
\end{itemize}
\end{proposition}
\begin{proof}
After enlarging, when necessary, the underlying probability space, it is possible to define, for all $\lambda \in (0,1)$ such that 
there exist several consecutive indices $k$ satisfying $\lambda_k = \lambda$, a uniform random permutation $\tau_{\lambda}$ of the set of these indices, 
for all $k \geq 1$ such that $\lambda_k > 0$, 
a uniform random variable $U_k$ on $C_k$, and for all $n \geq 1$, a uniform variable $V_n$ on $L$, in a way such that the virtual permutation $\sigma$, the 
variables $(U_k)_{k \geq 1}$, $(V_n)_{n \geq 1}$ and all the permutations of the form $\tau_{\lambda}$ are independent. Now, let $(x_n)_{n \geq 1}$ be an enumeration of $E$. By the results on 
partitions by Kingman, it is almost surely possible to define a sequence $(k_n)_{n \geq 1}$ of integers by induction, as follows:
\begin{itemize}
\item If $x_n$ is a fixed point of $\sigma$, then $k_n= 0$.
\item If $x_n$ is equivalent to $x_m$ for some $m < n$, then $k_n = k_m$, independently of the choice of the index $m$.
\item If $x_n$ is not a fixed point of $\sigma$ and is not equivalent to $x_m$ for any $m < n$, then $k_n$ is an integer such that 
$\lambda(x_n) = \lambda_{k_n}$, and if this condition does not determine $k_n$ uniquely, then this integer is chosen in a way such that $k_n$  is different from $k_m$ for 
all $m < n$ and $\tau_{\lambda(x_n)}(k_n)$ is as small as possible. 
\end{itemize}
\noindent 
The random permutations of the form $\tau_{\lambda}$ are used in order to guarantee the symmetry between all the circles in $C$ which have the same
perimeter. Moreover, one checks that for all $m, n \geq 1$ such that $x_m$ and $x_n$ are not fixed points of $\sigma$,  $\lambda(x_n) = \lambda_{k_n}> 0$, and 
$x_m \sim x_n$ if and only if $k_m = k_n$. Let us now define the variables $(X_{x_n})_{n \geq 1}$ by induction, as follows:
\begin{itemize}
\item If $k_n = 0$, then $X_{x_n} = V_n$.
\item If $k_n > 0$ and $k_n \neq k_m$ for all $m < n$, then $X_{x_n} = U_{k_n} \in C_{k_n}$.
\item If $k_n > 0$, if $k_n = k_m$ for some $m < n$, and if $m_0$ denotes the smallest possible value of $m$ satisfying this equality, then
 $X_{x_n}$ is the unique point of $C_{k_n} = C_{k_{m_0}}$ such that 
$X_{x_n} - U_{k_{m_0}} = \delta(x_{m_0},x_n)$, modulo $\lambda_{k_n} = \lambda_{k_{m_0}}$. 
\end{itemize}
\noindent
For all $n \geq 1$, $X_{x_n} \in L$ if and only if $k_n = 0$, otherwise, $X_{x_n} \in C_{k_n}$. Moreover, 
if $k_m = k_n > 0$ for some $m, n \geq 1$, and if $m_0$ is the smallest index such that $k_m = k_{m_0}$, 
then modulo $\lambda_{k_m}$, almost surely, $$X_{x_n} - X_{x_m} = (X_{x_n} - U_{k_{m_0}}) -  (X_{x_m} - U_{k_{m_0}})
= \delta(x_{m_0}, x_n) -  \delta(x_{m_0}, x_m) = \delta(x_m,x_n).$$
One deduces that Proposition \ref{circles} is satisfied, provided that the variables $(X_{x_n})_{n \geq 1}$ are independent
and uniform on $C$. 
In order to show this fact, let us first observe that by Proposition \ref{uniform}, the following holds: for all $n \geq 1$, conditionally on 
the restriction of the equivalence relation $\sim$ to the set $\{x_1, \dots, x_n\}$, the variables $\Delta(x_{m_0}, x_m)/ \lambda(x_{m_0})$, for 
$m_0 < m \leq n$ and  $x_{m_0} \sim x_m$, where $x_{m_0}$ is the element of its equivalence class with the smallest index, are uniform on 
$[0,1]$, independent, and form a family  which is independent of $(\lambda(x_m))_{1 \leq m \geq n}$. Now, once the restriction 
of $\sim$ to $\{x_1, \dots, x_n\}$ is given, the sequence $(k_m)_{1 \leq m \leq n}$ 
is uniquely determined by the sequence $(\lambda(x_m))_{1 \leq m \geq n}$, and the permutations of the form $\tau_{\lambda}$ (which form 
a family independent of $\sigma$), and then it is independent of the family of variables 
$\Delta(x_{m_0}, x_m)/ \lambda(x_{m_0})= \Delta(x_{m_0}, x_m)/ \lambda_{k_{m_0}}$ stated above. Since the restriction of   
$\sim$ to $\{x_1, \dots, x_n\}$ is a function of the sequence  $(k_m)_{1 \leq m \leq n}$, one deduces that conditionally on this sequence,
the variables $(X_{x_m})_{m \geq 1}$ are independent, $X_{x_m}$ being uniform on $C_{k_m}$ if $k_m > 0$, and uniform on $L$ if 
$k_m = 0$. 
It is now sufficient to prove that the variables $(k_n)_{n \geq 1}$ are independent and that for all $k \geq 1$,
$\mathbb{P} [ k_n = k ] = \lambda_k$. By symmetry, for $k^{(1)},\dots,k^{(p)} \geq 1$, and for any distinct integers $n_1, \dots, n_p \geq 1$,
the probability that  $k_{n_j} = k^{(j)}$ for all $j \in \{1, \dots, p\}$, does not depend on $n_1, \dots, n_p$ (note that 
the permutations $\tau_{\lambda}$ play a crucial role for this step of the proof of Proposition \ref{circles}). 
Hence, for $m > p$,
$$\mathbb{P} [ \forall j  \in \{1, \dots, p\}, k_j = k^{(j)} ] 
= \frac{(m-p)!}{m!} \, \mathbb{E} \left[ \sum_{1 \leq n_1 \neq \dots \neq n_p \leq m} \mathds{1}_{\forall j  \in \{1, \dots, p\}, k_{n_j} = k^{(j)} } \right],$$
 and then:
 \begin{align*}
   \mathbb{E} \left[ \prod_{j=1}^p \frac{\left(|\{r \in  \{1, \dots m \}, k_r = k^{(j)} \}| - p \right)_+}{m} \, \right] 
  & \leq \mathbb{P} [ \forall j  \in \{1, \dots, p\}, k_j= k^{(j)}] 
  \\ & \leq \mathbb{E} \left[ \prod_{j=1}^p \frac{|\{r \in  \{1, \dots m \}, k_r = k^{(j)} \}| }{m-p}  \, \right] 
 \end{align*}
 \noindent
 By using Proposition \ref{loup37} and dominated convergence for $m$ going to infinity, one deduces:
 $$\mathbb{P} [ \forall j  \in \{1, \dots, p\}, k_j= k^{(j)}] = \prod_{j=1}^p \lambda_{k_j}.$$
\end{proof}
\noindent
Now, we observe that it is possible to recover the virtual permutation $\sigma$ from the variables $(X_x)_{x \in E}$. More precisely, 
one has the following:
\begin{proposition} \label{circles2}
Let us consider the setting and the notation of Proposition \ref{circles}. For all $I \in \mathcal{F}(E)$ and $x \in I$, the element $\sigma_I(x)$ can almost surely
be obtained as follows:
\begin{itemize}
\item If $X_x \in L$, then $\sigma_I(x) = x$.
\item If for $k \geq 1$, $X_x \in C_k$, then $\sigma_I(x) = y$, where $X_y$ is the first point of the intersection of $C_k$ and the set $\{X_z, z \in I\}$, 
 encountered by moving counterclockwise on $C_k$, starting just after $X_x$ (for example, if $X_x$ is the unique element of $C_k \cap \{X_z, z \in I\}$, 
 then $x$ is a fixed point of $\sigma_I$). 
\end{itemize}
\end{proposition}
\begin{proof}
If $x$ is a fixed point of $\sigma_I$, then either $X_x \in L$, or the unique $y \in I$ such that $X_y$ is on the same circle as $X_x$ is 
$x$ itself. If $x$ is not a fixed point of $\sigma_I$, then $X_{\sigma_I(x)}$ is on the same circle $C_k$ as $X_x$. Moreover, for any $y \in I$ 
different from $x$ and $\sigma_I(x)$, but in the same cycle of $\sigma_I$, one has $0 < \Delta(x,\sigma_I(x)) < \Delta(x,y)$. Hence, there is no
 point of $C_k \cap \{X_z, z \in I\}$ on the open circle arc coming counterclockwise from $X_x$ to $X_{\sigma_I(x)}$. 
\end{proof}
\noindent
A consequence of Proposition \ref{circles} and \ref{circles2} is the following: 
\begin{corollary} \label{circles3}
Let $E$ be a countable set, $\lambda$ a sequence in the simplex $\Lambda$, $C$ the set constructed in Proposition \ref{circles}, and 
$(X_x)_{x \in E}$ a sequence of i.i.d. random variables, uniform on $C$. For all $I \in \mathcal{F}(E)$, it is almost surely possible to 
define a permutation 
$\sigma_I \in \Sigma_I$, by the construction given in Proposition \ref{circles2}. Moreover, $(\sigma_I)_{I \in E}$ is almost surely a virtual 
permutation, and its distribution is the measure $\mathbb{P}_{\lambda}$ defined in Proposition \ref{loup37}. 
\end{corollary}
\begin{proof}
The possibility to define almost surely $\sigma_I$ for all $I \in \mathcal{F}(E)$ comes from the fact that the points $(X_x)_{x \in E}$ are almost
surely pairwise distinct. Moreover, since $(\sigma_I)_{I \in \mathcal{F}(E)}$ is a deterministic function of the sequence of points $(X_x)_{x \in E}$, its 
distribution is uniquely determined by the assumptions of Proposition \ref{circles3}. Since Propositions \ref{circles} and \ref{circles2} give 
a particular setting on which $(\sigma_I)_{I \in \mathcal{F}(E)}$ is a virtual permutation following the distribution $\mathbb{P}_{\lambda}$, we 
are done. 
\end{proof}
\noindent
Proposition \ref{loup37} and Corollary \ref{circles3} give immediately the following description of all the central measures on $\Sigma_E$:
\begin{corollary} \label{circles4}
Let $\nu$ be a probability measure on $\Lambda$, $\lambda$ a random sequence following the distribution $\nu$, and $C$ the random set 
constructed from the sequence $\lambda$ as in Proposition \ref{circles}. Let $(X_x)_{x \in E}$ be a sequence of random points of 
$C$, independent and uniform conditionally on $\lambda$. For all $I \in \mathcal{F}(E)$, it is almost surely possible to 
define a permutation 
$\sigma_I \in \Sigma_I$, by the construction given in Proposition \ref{circles2}. Moreover, $(\sigma_I)_{I \in E}$ is almost surely a virtual 
permutation, and its distribution is the measure $\mathbb{P}_{\nu}$ defined in Proposition \ref{loup37}. 
\end{corollary}
\noindent 
This construction will be used in the next section, in order to study how  $\sigma$ acts on the completion of the space $E$ with respect to 
a random metric related to $\delta$. 
\section{A flow of transformations on a completion of $E$}  \label{s3}
\noindent
When one looks at the construction of the Ewens measures on the set of virtual permutations which is given at the end of Section \ref{s2}, it is natural to consider the random metric
given by the following proposition:
\begin{proposition} \label{distance}
Let $E$ be a countable set, and $\sigma$ a virtual permutation of $E$, which follows a central distribution. Let $d$ be a random function 
from $E^2$ to $\mathbb{R}_+$, almost surely defined as follows:
\begin{itemize}
\item If $x, y \in E$ and $x \not\sim_{\sigma} y$, then $d(x,y) = 1$;
\item If $x, y \in E$ and $x \sim_{\sigma} y$, then $d(x,y) = \inf\{|a|, \delta(x,y) = a \, (\operatorname{mod. } \lambda(x))\}$.
\end{itemize}
\noindent
Then, $d$ is almost surely a distance on $E$.
\end{proposition}
\begin{proof}
For all $x \in E$, $\delta(x,x) = 0$ almost surely, which implies that $d(x,x) = 0$. Conversely if $d(x,y) = 0$ for $x, y \in E$, then $x \sim y$ and $\delta(x,y) = 0$, 
which holds with strictly positive probability only for $x = y$. Finally if $x, y, z \in E$, then there are two cases:
\begin{itemize}
\item If $x, y, z$ are not equivalent, then $d(x,y)=1$ or $d(y,z) = 1$, and $d(x,z) \leq 1$, which implies that $d(x,y) + d(y,z) \leq d(x,z)$.
\item If $x \sim y \sim z$, then the triangle inequality holds because $\delta(x,z) = \delta(x,y) + \delta(y,z)$ almost surely. 
\end{itemize}
\end{proof}
\noindent
Once the metric space $(E,d)$ is constructed, it is natural to embed it in another space which is better known.
\begin{proposition} \label{isometry}
Let $E$ be a countable set, let $\sigma$ be a virtual permutation following a central measure, let $(\lambda_k)_{k \geq 1}$ be 
the non-increasing sequence of the asymptotic cycle lengths of $\sigma$, defined in Proposition \ref{loup37}, and let $C$ be 
the random space defined from $(\lambda_k)_{k \geq 1}$ as in Proposition \ref{circles}. Moreover, let us define the random metric $D$ on 
$C$ as follows:
\begin{itemize}
\item If $x, y$ are in $C_k$ for the same value of $k$, then $D(x,y) = \inf \{ |a|,  x-y = a \, (\operatorname{mod. } \lambda_k)\}$.
\item If $x= y \in L$, $D(x,y) = 0$.
\item Otherwise, $D(x,y) = 1$.
\end{itemize}
Then, if the probability space is sufficiently large  to guarantee the existence of the random variables $(X_x)_{x \in E}$ described in 
Proposition \ref{circles}, and if $K := \{ X_x, x \in E\}$, then the application $x \mapsto X_x$ is almost surely a bijective isometry from $(E,d)$ to $(K,D)$. 
\end{proposition}
\begin{proof}
Since the variables $(X_x)_{x \in E}$ are almost surely pairwise disjoint, the map $\phi : x \mapsto X_x$ is almost surely bijective from
$E$ to $K$. By comparing the definitions of $d$ and $D$ and by using the fact that $\delta(x,y) = X_y - X_x$ for $x \sim y$, we easily see
that $\phi$ is isometric. 
\end{proof}
\noindent 
The isometry defined in Proposition \ref{isometry} gives an intuitive idea on how the completion of $(E,d)$ looks like:
\begin{proposition} \label{extension}
Let us take the assumptions and the notation of Proposition \ref{isometry} and let us define the
 random metric space $(\widehat{E}, \widehat{d})$ as the completion of $(E,d)$. Then
the following properties hold almost surely:
\begin{itemize}
\item The map $\phi$ can be extended in a unique way to a bijective and isometric map $\widehat{\phi}$ from $(\widehat{E}, \widehat{d}\,)$ to $(H,D)$,
where $$H := (K \cap L) \cup \, \bigcup_{k \geq 1} C_k.$$
\item There exists a unique extension $\widehat{\sim}_{\sigma}$ of the equivalence relation $\sim_{\sigma}$ to the set $\widehat{E}$, such that 
the set $\{(x, y) \in \widehat{E}^2, x \, \widehat{\sim}_{\sigma} y\}$ is closed in $\widehat{E}^2$, for the topology induced by the distance $\widehat{d}$.
\item For all $x$, $y$ distinct in $\widehat{E}$, $x \, \widehat{\sim}_{\sigma} y$ if and only if $\widehat{\phi}(x)$ and $\widehat{\phi}(y)$ are on 
a common circle $C_k$, for some $k \geq 1$.
\item There exists a unique map $\widehat{\lambda}$ from $\widehat{E}$ to $[0,1]$, extending $\lambda$ in a continuous way.
\item For all $x \in \widehat{E}$, $k \geq 1$, $\widehat{\lambda}(x) = \lambda_k$ if $\widehat{\phi}(x) \in C_k$ and $\widehat{\lambda}(x) = 0$ if 
$\widehat{\phi}(x) \in L$.
\item There exists a unique map $\widehat{\delta}$ from the set $\{(x, y) \in \widehat{E}^2, x \, \widehat{\sim}_{\sigma} y\}$ to $H$, extending $\delta$ to a continuous map, 
for the topologies induced by the distances $\widehat{d}$ and $D$;
\item For all $x$, $y$ distinct in $\widehat{E}$ such that  $x \, \widehat{\sim}_{\sigma} y$, $\widehat{\delta}(x,y) = \widehat{\phi}(y) - \widehat{\phi}(x)$, modulo
$\lambda_k$, if $\widehat{\phi}(x)$ and $\widehat{\phi}(y)$ are on the circle $C_k$. 
\end{itemize}
\end{proposition}
 \begin{proof}
 Since, conditionally on $(\lambda_k)_{k \geq 1}$ the variables $(X_x)_{x \in E}$ are independent and uniform on $C$, the 
 space $(K,D)$ is almost surely dense in $(H,D)$, which implies the first item.
 The description of the third item proves the existence of $\widehat{\sim}_{\sigma}$: its uniqueness comes from the fact that
  $\sim_{\sigma}$ is defined on a dense subset of $\widehat{E}^2$. Similarly, the existence of $\widehat{\lambda}$ is deduced from the description given 
  in the fifth item, and its uniqueness comes from the density of $E$ in $\widehat{E}$. Finally,
  the uniqueness of $\widehat{\delta}$ comes from the density of 
 the set $\{x, y \in E, x \, \sim_{\sigma} y\}$ in the set $\{(x, y) \in \widehat{E}^2, x \, \widehat{\sim}_{\sigma} y\}$, and its existence is due to
 the continuity, for all $k \geq 1$, of the map $(x,y) \mapsto \widehat{\phi}(y) - \widehat{\phi}(x)$ from 
 $\{(x, y) \in \widehat{E}^2, \widehat{\phi}(x), \widehat{\phi}(y) \in C_k\}$ to $\mathbb{R}/ \lambda_k \mathbb{Z}$. 
 \end{proof}              
 \begin{remark}
 If $L$ is empty, which happens if and only if $\sum_{k \geq 1} \lambda_k = 1$ (for example under Ewens measure of any parameter), then $H$ is equal to $C$. Otherwise,
 $H$ is the union of the circles included in $C$, and a discrete countable set.       
 \end{remark}                                                                                          
 \noindent
From now, the notations $\widehat{\phi}$, $\widehat{\lambda}$, $\widehat{\delta}$,  $\widehat{\sim}_{\sigma}$, will be replaced by 
$\phi$, $\lambda$, $\delta$, $\sim_{\sigma}$ (or $\sim$), since this simplification is consistent with the previous notation.       
Note that the map $\phi$ is not determined by $\sigma$, since it depends to the choice of the variables $(X_x)_{x \in E}$ in Proposition \ref{circles}, which is 
not unique in general (for example, if $\lambda_1 > 0$ almost surely, and if $a \in \mathbb{R}$ is fixed, then one can replace each point $X_x \in C_1$ by the unique point
 $X'_x \in C_1$ such that $X'_x - X_x = a$, modulo $\lambda_1$). However, as stated in Proposition
\ref{extension}, the extensions of $\sim$, $\lambda$ and $\delta$ are almost surely uniquely determined. We now have  all the ingredients needed to construct the flow 
of transformations on $\widehat{E}$ indicated in the title of this section. 
\begin{proposition} \label{alpha496}
With the notation above and the assumptions of Proposition \ref{isometry}, there exists almost surely a 
unique family $(S^{\alpha})_{\alpha \in \mathbb{R}}$ of bijective isometries of the set $\widehat{E}$, such 
that for all $x \in \widehat{E}$, $S^{\alpha}(x) \sim x$ and in the case where $x$ is not a fixed point 
of $\sigma$, $\delta(x, S^{\alpha}(x)) = \alpha$ modulo $\lambda(x)$. Almost surely, this family is given 
as follows: 
\begin{itemize}
\item If $x$ is a fixed point of $\sigma$, then $S^{\alpha}(x) = x$. 
\item If $x$ is not a fixed point of $\sigma$, then $S^{\alpha}(x)$ is the unique point of $\widehat{E}$ 
such that $\phi(S^{\alpha}(x))$ is on the same circle as $\phi(x)$, and $\phi(S^{\alpha}(x)) - \phi(x) = \alpha$ modulo $\lambda(x)$. 
\end{itemize}
\noindent
Moreover, almost surely, $S^{\alpha+\beta} = S^{\alpha}S^{\beta}$ for all $\alpha, \beta \in \mathbb{R}$.
\end{proposition}
\begin{proof}
For all $x, y \in \widehat{E}$, $\alpha \in \mathbb{R}$, let us denote by $C(x,y,\alpha)$ the
 condition described as follows:
\begin{itemize}
\item If $x$ is a fixed point of $\sigma$, then $y=x$.
\item If $x$ is not a fixed point of $\sigma$, then $x \sim y$ and $\phi(y)- \phi(x) = \alpha$, 
modulo $\lambda(x)$.
\end{itemize}
\noindent
It is clear that the condition $C(x,y,\alpha)$ determines uniquely $y$ once $x$ and $\alpha$ are 
fixed, which proves that almost surely, the map $S^{\alpha}$ from $\widehat{E}$ to $\widehat{E}$ 
is well-defined for all $\alpha \in \mathbb{E}$, and that its explicit description is given in
Proposition \ref{alpha496}. By using this description, it is immediate to deduce 
that for all $\alpha, \beta \in \mathbb{R}$, $S^{\alpha}$ is isometric
 and $S^{\alpha+\beta} = S^{\alpha}S^{\beta}$. In particular,
 $S^{\alpha}S^{-\alpha} = S^0$ is the identity
 map of $\widehat{E}$, and then $S^{\alpha}$ is bijective.
\end{proof} 
\noindent
Now, as written in the introduction, the flow $(S^{\alpha})_{\alpha \in \mathbb{R}}$ can be
 seen as the limit, for large sets $I \in \mathcal{F}(E)$, of a power of $\sigma_I$, with exponent
 approximately equal to $\alpha |I|$. The following statement gives a rigorous meaning of this 
idea. 
\begin{proposition} \label{convergence}
Let us take the assumptions and the notation above, and let us define a 
sequence $(\alpha_n)_{n \geq 1}$ in $\mathbb{R}$, such that $\alpha_n/n$ tends to a limit
 $\alpha$ when $n$ goes to infinity. Then, for all $\epsilon > 0$, there exist
 $C(\epsilon), c(\epsilon) > 0$ such that for all $I \in \mathcal{F}(E)$:
$$\mathbb{P} \left[   \exists x \in I, \, d( \sigma_I^{\alpha_{|I|}} (x), S^{\alpha} (x)) \geq 
\epsilon \right] \leq C(\epsilon) e^{- c(\epsilon)|I|}.$$
 \end{proposition} 
 \begin{proof}
  Let us suppose $\alpha> 0$. Let $x \in E$, let $I$ be a finite subset of $E$ containing $x$, and let $A$
be the set of $y \in I$ which are equivalent to $x$ but different from $x$. Then, by
 Proposition \ref{uniform}, conditionally on $A$, and on the event where this set is not empty,
 the variables $\delta(x,y)/\lambda(x)$ for
 $y \in A$ are independent, uniform on $\mathbb{R}/\mathbb{Z}$ and they from of
family which is independent of $\lambda(x)$. Moreover, if $(b_k)_{k \in \mathbb{Z}}$ denotes 
the increasing family of the reals in the class of $\delta(x,y)/\lambda(x)$ modulo $1$, for 
some $y \in A \cup \{x\}$, with $b_0$ equal to zero, then $\delta(x,\sigma_I^{\alpha_{|I|}}(x))/\lambda(x)$
 is the class of $b_{\alpha_{|I|}}$ modulo $1$. One deduces that
conditionally on $|A|$, this cardinality being different from zero, 
 $$\delta(S^{\alpha}(x),\sigma_I^{\alpha_{|I|}}(x)) = \lambda(x) b(|A|, \alpha_{|I|}) - \alpha,$$
where $(b(|A|, k))_{k \in \mathbb{Z}}$ is  a $1$-periodic random increasing family of reals, independent of $\lambda(x)$, and such that its elements in $[0,1)$
 are $b(|A|,0) = 0$, and $|A|$ independent variables, uniform on $(0,1)$. One deduces that for 
$|A| \geq 1$, 
 $$d(S^{\alpha}(x),\sigma_I^{\alpha_{|I|}}(x)) \leq \left| 
\lambda(x) b(|A|, \alpha_{|I|}) - \alpha \right| \wedge \lambda(x).$$
Moreover, one has obviously, in any case:
$$d(S^{\alpha}(x),\sigma_I^{\alpha_{|I|}}(x))  \leq \lambda(x).$$
 One deduces, for all $\epsilon > 0$:
$$
 \mathbb{P} \left[ d(S^{\alpha}(x),\sigma_I^{\alpha_{|I|}}(x)) \geq \epsilon \right]
 \leq \mathbb{P} \left[ \left| \frac{\alpha_{|I|}}{|I|} - \alpha \right|  \geq \epsilon/3 \right]  \nonumber
  + \mathbb{P} \left[ \lambda(x) \geq \epsilon, \, \left| \frac{\lambda(x) \alpha_{|I|}}{|A|} - \frac{\alpha_{|I|}}{|I|} \right| \geq \epsilon/3 \right] \nonumber
 $$
\begin{equation}
+ \, \mathbb{P} \left[ \lambda(x) \geq \epsilon, |A| \geq 1, \,  \left| \frac{\lambda(x) \alpha_{|I|}}{|A|} - \alpha \right| \leq 2 \epsilon/3, 
  \,  \lambda(x) \left|  b(|A|, \alpha_{|I|}) - \frac{\alpha_{|I|}}{|A|} \right| \geq \epsilon/3 \right].
 \label{S}
 \end{equation}
 \noindent
In \eqref{S}, the event involved in the second term of the sum is always supposed to occur if $A$ is 
empty. Now, for $|I|$ large enough,
\begin{equation}
\mathbb{P} \left[ \left| \frac{\alpha_{|I|}}{|I|} - \alpha \right|  \geq \epsilon/3 \right]  = 0, 
\label{nnbar1}
\end{equation}
since $\alpha_n/n$ is deterministic and tends to $\alpha$ when $n \geq 1$ goes to infinity. Moreover, 
\begin{align*}
\mathbb{P} \left[ \lambda(x) \geq \epsilon, \,  \left| \frac{\lambda(x) \alpha_{|I|}}{|A|} - \frac{\alpha_{|I|}}{|I|} \right| \geq \epsilon/3 \,   \right] & 
= \mathbb{P} \left[ \lambda(x) \geq \epsilon, \, \left| |A| - \lambda(x) |I| \, \right| \geq  \frac{\epsilon  |A| |I|}{3 |\alpha_{|I|}|} \right] \\ & 
\leq \mathbb{P} \left[  \lambda(x) \geq \epsilon,  \, \left| |A| - \lambda(x) |I| \, \right| \geq  \frac{\epsilon^2 |I|^2}{6 |\alpha_{|I|}|} \right] \\ 
& + \mathbb{P} \left[ \lambda(x) \geq \epsilon, \, |A| \leq |I| \epsilon/2 \right]  \\ & 
\leq 2 \, \mathbb{P} \left[  \left| |A| - \lambda(x) |I| \, \right| \geq  \frac{\epsilon^2 |I|}{6 \alpha + 2 \epsilon + 1} \right] ,
\end{align*}
\noindent
if $|I|$ is large enough (depending only on the sequence $(\alpha_n)_{n \geq 1}$). Now, conditionally on $\lambda(x)$, $|A|$ is the sum of $|I|-1$ independent
 Bernoulli random variables, with parameter $\lambda(x)$. Hence, there exist $c_1, c_2 >0$, depending only on $(\alpha_n)_{n \geq 1}$ and $\epsilon$, 
 such that:
 \begin{equation}
\mathbb{P} \left[ \lambda(x) \geq \epsilon, \,  \left| \frac{\lambda(x) \alpha_{|I|}}{|A|} - \frac{\alpha_{|I|}}{|I|} \right| \geq \epsilon/3 \,   \right] \leq 
 c_1 e^{-c_2 |I|}. \label{nnbar2}
\end{equation}
 In order to evaluate the last term of \eqref{S}, let us observe that if the corresponding event holds, then for $|I|$ large enough,
 $\alpha_{|I|} > |I| \alpha/2$, which implies:
 \begin{equation}
 |A| \geq  \frac{\lambda(x) \alpha_{|I|}}{\alpha + 2 \epsilon/3} \geq \frac{\epsilon \alpha |I|}{ 2 (\alpha + \epsilon)}. \label{kd28}
 \end{equation}
 Moreover, one has:
 $$ b(|A|, \alpha_{|I|}) = \beta + [\alpha_{|I|}/(|A|+1) ],$$
  where the brackets denote the integer part, and where, conditionally on $|A|$, $\beta$ is a beta random variable of parameters
  $k := \alpha_{|I|} - (|A|+1) [\alpha_{|I|}/(|A|+1) ] $ and $(|A|+1) - k$. One deduces that, conditionally on $|A|$,
  the probability that $|b(|A|, \alpha_{|I|}) - \alpha_{|I|}/(|A|+1)|$ is greater than or equal to $\epsilon/6$ decreases
  exponentially with $|A|$, independently of $\alpha_{|I|}$. Moreover, for $|I|$ large enough, by \eqref{kd28}:
  $$\left| \frac{\alpha_{|I|}}{|A|} -  \frac{\alpha_{|I|}}{|A| +1} \right| \leq \frac{\alpha_{|I|}}{|A|^2} \leq \frac{ 5 (\alpha+\epsilon)^2}
  { \epsilon^2 \alpha |I|} \leq \epsilon/7.$$
  One deduces that there exist $c_3, c_4 >0$, depending only on $(\alpha_n)_{n \geq 1}$ and $\epsilon$, such that: 
  \begin{equation}
\mathbb{P} \left[ \lambda(x) \geq \epsilon,  \,  \left| \frac{\lambda(x) \alpha_{|I|}}{|A|} - \alpha \right| \leq 2 \epsilon/3, 
  \,  \lambda(x) \left|  b(|A|, \alpha_{|I|}) - \frac{\alpha_{|I|}}{|A|} \right| \geq 
\epsilon/3 \right] \leq c_3 e^{-c_4} |I|, \label{nnbar3}
\end{equation}
  By \eqref{nnbar1}, \eqref{nnbar2}, \eqref{nnbar3}, there exist $c_5, c_6 > 0$ such that 
  $$\mathbb{P} \left[ d(S^{\alpha}(x),\sigma_I^{\alpha_{|I|}}(x)) \geq \epsilon \right] \leq c_5 e^{-c_6 |I|}.$$
  By adding these estimates for all $x \in I$, one deduces Proposition \ref{convergence} for $\alpha> 0$. The proof 
  is exactly similar for $\alpha < 0$. Now, let $(\alpha_n)_{n \geq 1}$ and 
  $(\beta_n)_{n \geq 1}$ be two sequences such that $\alpha_n/n$ and $\beta_n/n$ tend to $1$. Then,
  \begin{align*}
  \sup_{x \in I} d(x, \sigma_I^{\alpha_{|I|} - \beta_{|I|}}(x)) & 
  \leq \sup_{x \in I} d(x, S^{1} (\sigma_I^{ - \beta_{|I|}}(x))) +
  \sup_{x \in I} d(S^{1} (\sigma_I^{ - \beta_{|I|}}(x)), \sigma_I^{\alpha_{|I|} - \beta_{|I|}}(x)) \\ 
  & \leq \sup_{x \in I} d(S^{-1}(x), \sigma_I^{ - \beta_{|I|}}(x)) + 
  \sup_{x \in I} d(S^{1} (x), \sigma_I^{\alpha_{|I|}}(x)),
  \end{align*}
  \noindent
  since $S^{1}$ is an isometry of $\widehat{E}$ and $\sigma_I^{ - \beta_{|I|}}$ is a bijection of $|I|$. One deduces that Proposition \ref{convergence} holds also for $\alpha = 0$.
 \end{proof}         
 \begin{corollary}
 With the assumptions of Proposition \ref{convergence}, if $|I|$ goes to infinity, then the supremum of $d( \sigma_I^{\alpha_{|I|}} (x), S^{\alpha} (x))$ for $x \in I$ converges to zero 
 in probability, in $L^p$ for all $p \in [1, \infty)$, and almost surely along any deterministic, strictly increasing sequence $(I_n)_{n \geq 1}$ of sets in $\mathcal{F}(E)$. In particular, if 
 $(x_n)_{n \geq 1}$ is a random sequence of elements in $E$, such that $x_n \in I_n$ for all $n \geq 1$ and $x_n$ converges almost surely 
 to a random limit $x \in \widehat{E}$ when $n$ goes to infinity (this situation holds if $x \in E$ and $x_n = x$ for $n$ large enough), then $\sigma_{I_n}^{\alpha_{|I_n|}} (x_n)$ converges almost surely to $S^{\alpha}(x)$ when $n$ goes to infinity. 
  \end{corollary}
  \begin{proof}
  The convergence in probability is directly implied by Proposition \ref{convergence}, and
 it implies convergence in $L^p$ for all $p \in [1, \infty)$, since the distance is bounded by one. 
  The almost sure convergence is proven by using Borel-Cantelli lemma. 
  \end{proof}                                                                    
  \noindent
  We have now constructed a flow of transformations of $\widehat{E}$ and we have related it in a rigorous way to the iterations of $\sigma_I$ for large sets $I$. In Section
  \ref{s4}, we interpret this flow as a flow of operators on a suitable random functional
 space, and we construct its infinitesimal generator, as discussed in the introduction. 
 \section{A flow of operators on a random functional space} \label{s4}
 Let us now define the random functional space on which the flow of operators described below acts.
 We first take, as before, a random virtual permutation $\sigma$ on a countable set $E$, which 
follows a central probability measure.
By the results given in the previous sections, the following events hold almost surely:
 \begin{itemize}
 \item The variables $\lambda(x)$ are well-defined, strictly 
positive for all $x \in E$ which are not fixed points of $\sigma$, and 
$\lambda(x) = \lambda(y)$ for all $x, y \in E$ such that $x \sim_{\sigma} y$.
 \item With the definitions of Proposition \ref{loup37}, the 
non-increasing sequence $(\lambda_k)_{k \geq 1}$ of the cycle lengths is an element of 
 the simplex $\Lambda$.
 \item The quantity $\delta(x,y)$ exists for all $x, y \in E$ such that $x \sim_{\sigma} y$,
 and $\delta(x,x) = 0$, $\delta(x,y) = - \delta(y,x)$,
 $\delta(x,y) + \delta(y,z) = \delta(x,z)$ for all $x, y, z$ such that 
$x \sim_{\sigma} y \sim_{\sigma} z$.
 \item There exists a bijective isometry $\phi$ from $E$ to a dense subset of 
$$H := L_0 \cup \, \bigcup_{k \geq 1, \lambda_k > 0} C_k,$$
where $C_k$ is a circle of perimeter $\lambda_k$, the set $L_0$ is
empty if $\sum_{k \geq 1} \lambda_k = 1$ and countable if $\sum_{k \geq 1} \lambda_k < 1$,
and the union is disjoint,
 for the distances $d$ and $D$ defined similarly as in Propositions \ref{distance} and \ref{isometry}.
 \item The map $\lambda$, the distance $d$ and the equivalence
 relation $\sim_{\sigma}$ extend in a unique continuous way to the completed space $\widehat{E}$
 of $E$, for the distance $d$.
\item The map $\delta$ extends in a unique continuous way to the set 
$\{ (x,y) \in \widehat{E}^2, x \sim_{\sigma} y \}$.
 \item The isometry $\phi$ extends in a unique way to a bijective isometry from $(\widehat{E},d)$ to $(H,D)$.
\item For all distinct $x, y \in \widehat{E}$, $x \sim_{\sigma} y$ if and only if $\phi(x)$ and $\phi(y)$ are
 on the same circle included in $H$. 
\item For all distinct $x, y \in \widehat{E}$ such that $x \sim_{\sigma} y$,
 $\delta(x,y) = \phi(y) - \phi(x)$, modulo $\lambda(x)$.
 \item There exists a unique flow $(S^{\alpha})_{\alpha \in \mathbb{R}}$ of isometric bijections 
of $H$ such that for all $\alpha \in \mathbb{R}$, 
$x \in \widehat{E}$, $\phi(S^{\alpha}(x)) -\phi(x) = \alpha$ modulo $\lambda(x)$ if $x$ 
is not a fixed point of $\sigma$, and $S^{\alpha}(x) = x$ if $x$ is a fixed point of $\sigma$.
\item For all $\alpha, \beta \in \mathbb{R}$, $S^{\alpha}S^{\beta} = S^{\alpha + \beta}$.
 \end{itemize}
 From now, let us fix $\sigma$ such that all the items above are satisfied: no randomness is involved in the construction of the operator $U$ given below. 
 We can interpret the flow $(S^{\alpha})_{\alpha \in \mathbb{R}}$ as a flow of operators
 on a functional space defined on $E$. The first step in the corresponding construction
is the following result:
  \begin{proposition} \label{51}
Let $f$ be a function from $E$ to $\mathbb{C}$. If $f$ can be extended to a continuous function from $\widehat{E}$ to $\mathbb{C}$, then this extension is unique:
in this case we say that $f$ is continuous. For example, if $f$ is uniformly continuous from $E$ to $\mathbb{C}$, then it is continuous and its continuous extension to $\widehat{E}$ is 
also uniformly continuous.
 \end{proposition}
 \noindent
 When a function from $E$ to $\mathbb{C}$ is continuous, we can use its extension to $\widehat{E}$ in order to make the flow $(S^{\alpha})_{\alpha \in \mathbb{R}}$
 acting on it, as follows:
\begin{proposition} \label{y6}
In the setting above, one can define a unique flow $(T^{\alpha})_{\alpha \in \mathbb{R}}$ of linear operators on the space of continuous functions from $E$ to $\mathbb{C}$, satisfying 
the following properties:
\begin{itemize}
\item For any continuous function $f$ from $E$ to $\mathbb{C}$, $T^{\alpha} (f) (x) = \widehat{f} (S^{\alpha} (x))$ for all $\alpha \in \mathbb{R}$, where $\widehat{f}$ is the continuous extension of $f$ to $\widehat{E}$;
\item For all $\alpha, \beta \in \mathbb{R}$, $T^{\alpha + \beta} = T^{\alpha}T^{\beta} $.
\end{itemize} 
\end{proposition}
\noindent
\begin{proof}
Let $f$, $g$, $h$ be three continuous functions from $E$ to $\mathbb{C}$, such that $h = f + rg$ for some $r \in \mathbb{C}$. 
Let $\widehat{f}$, $\widehat{g}$ and $\widehat{h}$ be their continuous extensions to $\widehat{E}$. By uniqueness, one has
$\widehat{h} = \widehat{f} + r \widehat{g}$. Let us now define, for $\alpha \in \mathbb{R}$, 
three functions $f_{\alpha}$, $g_{\alpha}$ and $h_{\alpha}$ from $\widehat{E}$ to $\mathbb{C}$, by:
$$f_{\alpha}(x)  =  \widehat{f}(S^{\alpha} (x)),$$
$$g_{\alpha}(x)  =  \widehat{g}(S^{\alpha} (x)),$$
and
$$h_{\alpha}(x)  =  \widehat{h}(S^{\alpha} (x)),$$
for all $x \in \widehat{E}$. Since $S^{\alpha}$ is an isometry of $\widehat{E}$, the functions $f_{\alpha}$, $g_{\alpha}$ and $h_{\alpha}$ are continuous 
and satisfy: $h_{\alpha} = f_{\alpha} + r g_{\alpha}$. By taking their restrictions to $E$, one
 deduces the existence and the linearity of an operator $T^{\alpha}$, 
from the space of continuous functions on $E$ to itself, which satisfies the first item. Of course this operator is uniquely determined. Now, let $\beta \in \mathbb{R}$.
One has, for all $x \in E$:
$$T^{\alpha} T^{\beta} (f)(x)= \widehat{T^{\beta}(f)}(S^{\alpha} (x)),$$
and then,
$$T^{\alpha} T^{\beta} (f)(x)= \widehat{f}(S^{\beta} S^{\alpha}(x)) = \widehat{f}(S^{\alpha+\beta}(x)) = T^{\alpha+\beta}(f)(x).$$
\end{proof}
\noindent
 Let us now define the space of continuously differentiable functions with respect
 to the flow of operators $(T_{\alpha})_{\alpha \in \mathbb{R}}$:
 \begin{definition} \label{babel} 
 In the setting above, let $f$ be a continuous function from $E$ to $\mathbb{C}$. We say that $f$ is continuously differentiable, if and only if there exists
 a continuous function $Uf$ from $E$ to $\mathbb{C}$, necessarily unique, such that for all $x \in \widehat{E}$,
 $$\frac{\widehat{T^{\alpha}(f)}(x) - \widehat{f}(x)}{\alpha} \underset{\alpha \rightarrow 0}{\longrightarrow} \widehat{Uf}(x),$$
 where $\widehat{f}$, $\widehat{T^{\alpha}(f)}$ and $\widehat{Uf}$ are the continuous extensions of $f$, $T^{\alpha}(f)$ and $Uf$ to $\widehat{E}$. 
 \end{definition}
 \noindent 
The following result is immediate:
\begin{proposition}
The application $f \mapsto Uf$ from the space of continuously differentiable functions
 to the space of continuous functions from $E$ to $\mathbb{C}$, constructed  
in Definition \ref{babel}, is a linear operator. 
\end{proposition}
\noindent
Once the operator $U$ is defined, it is natural to study its eigenfunctions and eigenvalues. 
The following result holds:
\begin{proposition} \label{eigenfunctions}
The eigenvalues of $U$ are zero, and all the nonzero multiples of $2 i \pi / \lambda_k$ for 
$\lambda_k > 0$. The corresponding eigenspaces are described as follows:
\begin{itemize}
\item The space corresponding to the eigenvalue zero consists of all the functions $f$ from 
$E$ to $\mathbb{C}$ 
such that $x \sim_{\sigma} y$ implies $f(x) = f(y)$.
\item The space corresponding to the eigenvalue $ai$ for $a \in \mathbb{R}^*$ consists of 
all the functions $f$ such that $f(x)= 0$ if $\lambda(x) = 0$ or $\lambda(x)$ is not a multiple
of $2 \pi / a$, and such that for $\lambda(x)$ nonzero and divisible by $2 \pi/a$,
 the restriction of 
$f$ to the equivalence class of $x$ for $\sim_{\sigma}$ is proportional to 
$y \mapsto e^{ai \delta(x,y)}$.
\end{itemize}
\noindent
Consequently, the dimension of the space corresponding to the eigenvalue zero is equal to 
the number of indices $k \geq 1$ such that $\lambda_k > 0$ if $\sum_{k \geq 1} \lambda_k = 1$, 
and to infinity if $\sum_{k \geq 1} \lambda_k < 1$. Moreover, for $a \in \mathbb{R}^*$, the 
dimension of the space corresponding to the eigenvalue $ia$ is equal to the number of indices 
$k \geq 1$ such that $\lambda_k$ is a nonzero multiple of $2 \pi/a$, in particular it is finite.
\end{proposition}
\begin{proof}
Let $f$ be an eigenfunction of $U$ for an eigenvalue $b \in \mathbb{C}$, and let
  $\widehat{f}$ be its extension to $\widehat{E}$. One has $Uf = bf$, and 
then by continuity, $\widehat{Uf} = b \widehat{f}$. For all $x \in E$, 
let $g_x$ be the function from $\mathbb{R}$ to $\mathbb{C}$, given by:
$$g_x(\alpha) := \widehat{f}(S^{\alpha}(x)).$$
For all $\alpha, \beta \in \mathbb{R}$,
$$g_x(\alpha + \beta) = \widehat{f}(S^{\alpha}S^{\beta} (x)) = 
\widehat{T^{\beta}(f)}(S^{\alpha}(x)),$$
and then, for $\beta \neq 0$,
$$\frac{g_x(\alpha + \beta) - g_x(\alpha)}{\beta} = \frac{\widehat{T^{\beta}(f)} (S^{\alpha}(x))
- \widehat{f}(S^{\alpha}(x)) }{\beta} \underset{\beta \rightarrow 0}{\longrightarrow}
\widehat{Uf}(S^{\alpha}(x)) = b \, \widehat{f}(S^{\alpha}(x)) = b \, g_x(\alpha).$$ 
Hence, $g_x$ is continuously differentiable and satisfies the differential equation
$g_x'= b \, g_x$, which implies that $g_x$ is proportional to 
the function $\alpha \rightarrow e^{b \, \alpha}$. Since $S^{\alpha}(x) = x$ for 
$\alpha = \lambda(x)$, and for all $\alpha \in \mathbb{R}$ if $\lambda(x) = 0$, 
 one has $\lambda(x) > 0$ and $e^{b \, \lambda(x)} = 1$, $b = 0$ and 
$g_x$ constant, or $g_x$ identically zero. Therefore, one of the 
three following possibilities holds for all $x \in E$:
\begin{itemize}
\item $f$ is identically zero on the cycle of $x$.
\item $b= 0$ and $f$ is constant on the cycle of $x$.
\item $\lambda(x) > 0$, $b$ is multiple
of $2 i \pi/ \lambda(x)$ and the restriction of $f$ to the cycle of $x$ is 
proportional to $y \mapsto e^{b \, \delta(x,y)}$. 
\end{itemize}
\noindent
Conversely, it is easy to check that any function which satisfies one of the three items above 
for all $x \in E$ is an eigenfunction of $U$ for the eigenvalue $b$, which completes the 
proof of Proposition \ref{eigenfunctions}.
\end{proof}
\noindent
As discussed before, the operator $T^{\alpha}$ can be viewed as a limit of $\sigma_I^{\alpha_{|I|}}$ for
 large $I \in \mathcal{F}(E)$ and 
$\alpha_{|I|}$ equivalent to $\alpha |I|$. It is then natural to relate
 the permutation $\sigma_I$ to the operator $T^{1/|I|}$, and then the operator 
$|I| (\sigma_I - \operatorname{Id})$ to $|I| (T^{1/|I|} - \operatorname{Id})$, where $\sigma_I$ is
 identified with a permutation matrix. Now, the eigenvalues of 
$|I| (\sigma_I - \operatorname{Id})$ are equal to $|I| (e^{i \kappa} - 1)$, where $\kappa$ is an eigenangle
 of $\sigma_I$, and this quantity
is expected to be close to $i \kappa |I|$, on the other hand, $|I| (T^{1/|I|} - \operatorname{Id})$ is 
expected to be close to $U$. Hence, it is natural to compare the renormalized eigenangles 
of $\sigma_I$ (i.e. multiplied by $i |I|$), and the eigenvalues of $U$ computed 
in Proposition \ref{eigenfunctions}. The rigorous statement corresponding to this idea is 
the following:
\begin{proposition} \label{eigenangles}
Let $\sigma$ be a random virtual permutation on a countable set $E$, following
 a central measure. Let $X$ be the 
set of the eigenvalues of the random operator $iU$ (which is almost surely well-defined), and for 
$I \in \mathcal{F}(E)$, let $X_I$ be the set of the eigenangles of $\sigma_I$, multiplied by $|I|$. If
 $\gamma \in X$ (resp. $\gamma \in X_I$), let $m(\gamma)$ (resp. $m_I(\gamma)$) be the multiplicity 
of the corresponding eigenvalue (resp. rescaled eigenangle). Then $X$ and $X_I$, $I 
\in \mathcal{F}(E)$, are 
included in $\mathbb{R}$, and for all continuous functions $f$ from
$\mathbb{R}$ to $\mathbb{R}_+$, with compact support, the following convergence holds:
 $$\sum_{\gamma \in X_I} m_I(\gamma) f(\gamma) 
\underset{|I| \rightarrow \infty}{\longrightarrow} \sum_{\gamma \in X} m(\gamma) f(\gamma),$$
in probability, and almost surely along any fixed, strictly increasing sequence of sets in $\mathcal{F}(E)$.
\end{proposition}
\begin{proof}
Since $X$ and $X_I$ have no point in the interval $(-2 \pi, 2 \pi)$ except zero,
 it is sufficient to prove the convergence stated in Proposition \ref{eigenangles}
for $f=\mathds{1}_{\{0\}}$ and for $f$ nonnegative, continuous, with compact support and 
such that $f(0)= 0$. Let $(I_n)_{n \geq 1}$ be an increasing sequence of
 subsets of $E$, such that $|I_n| = n$. Let us suppose that the underlying probability space 
is large enough to apply Proposition \ref{circles}, and let us take the same notation.
 The multiplicity $m_{I_n}(0)$ is equal to 
the sum of the number of indices $k \geq 1$ such that there exists $x \in I_n$ with $X_x \in C_k$, 
and the number of elements $x \in I_n$ such that $X_x \in L$. By the fact that conditionally
on $(\lambda_k)_{k \geq 1}$, 
the variables $(X_x)_{x \in E}$ are independent and uniform on $C$, a weak form of the 
law of large numbers implies that $m_{I_n}(0)$ increases almost surely to the number of indices $k \geq 1$
such that $\lambda_k > 0$ if $\sum_{k \geq 1} \lambda_k = 1$, and to infinity otherwise.
In other words, $m_{I_n}(0)$ increases almost surely to $m(0)$, and then also in 
probability. Since the law of $\sigma$ is central, the 
convergence in probability holds also for $|I|$ going to infinity
and not only along the sequence $(I_n)_{n \geq 1}$, which proves the convergence in
 Proposition \ref{eigenangles} for $f=\mathds{1}_{\{0\}}$. Let us now suppose that 
 $f$ is nonnegative, continuous, with compact support and 
satisfies $f(0)= 0$.
One has almost surely, for all $n \geq 1$,
$$\sum_{\gamma \in X_{I_n}} m(\gamma) f(\gamma)  = \sum_{m \in \mathbb{Z} \backslash \{0\} } 
\sum_{k \geq 1} f( 2 \pi mn/ |I_n \cap \mathcal{C}_k|),$$
where $\mathcal{C}_k$ denotes the set of $x \in E$ such that $X_x \in C_k$, and with 
the convention:
$$f( 2 \pi mn/ |I_n \cap \mathcal{C}_k|) := 0$$ for $|I_n \cap \mathcal{C}_k| = 0$. Since $f$
has compact support, there exists $A > 0$ such that $f(t) = 0$ for $|t| \geq A$.
Hence, the condition $f(2 \pi mn/ |I_n \cap \mathcal{C}_k|) > 0$ implies 
that $2 \pi |m| < A$ and a fortiori $|m| \leq A$, on the other hand, it implies that 
$2 \pi n/ |I_n \cap \mathcal{C}_k| < A$, $|I_n \cap \mathcal{C}_k|/n \geq 1/A$, and in particular,
\begin{equation}
 \frac{1}{n} \, \left| I_n \cap \left( \bigcup_{l \geq k} \mathcal{C}_l \right) \right| \geq 1/A. \label{AAAAA}
 \end{equation}
Now, conditionally on $(\lambda_n)_{n \geq 1}$, the left hand side of \eqref{AAAAA}
has the same law as the average of $n$ i.i.d. Bernoulli random variables, with parameter
 $\sum_{l \geq k} \lambda_l$. Hence, by law of large numbers, if 
$k_0 \geq 1$ denotes the smallest integer such that $\sum_{l \geq k_0} \lambda_l \leq 1/2A$, one has
 almost surely, for $n$ large enough,
$f(2 \pi mn/ |I_n \cap \mathcal{C}_k|) = 0$ if $k \geq k_0$, and then
$$\sum_{\gamma \in X_{I_n}} m(\gamma) f(\gamma)  = \sum_{m \in (\mathbb{Z} \backslash \{0\}) \cap [-A,A] } \sum_{1 \leq k \leq k_0} f( 2 \pi mn/ |I_n \cap \mathcal{C}_k|).$$
By the continuity of $f$ and the fact that $|I_n \cap \mathcal{C}_k|/n$ tends to $\lambda_k$ when $n$ goes to infinity, one deduces that almost surely,
$$ \sum_{\gamma \in X_{I_n}} m(\gamma) f(\gamma)  \underset{n \rightarrow \infty}{\longrightarrow} \sum_{m \in (\mathbb{Z} \backslash \{0\}) \cap [-A,A] } \sum_{1 \leq k \leq k_0}  f( 2 \pi m/\lambda_k)
= \sum_{\gamma \in X} f(\gamma),$$
which gives the almost sure convergence stated in Proposition \ref{eigenangles}. By the centrality of 
the law of $\sigma$, one then deduces the convergence in probability.

\end{proof}

  \bibliographystyle{amsplain}
 
 \providecommand{\bysame}{\leavevmode\hbox to3em{\hrulefill}\thinspace}
\providecommand{\MR}{\relax\ifhmode\unskip\space\fi MR }
\providecommand{\MRhref}[2]{%
  \href{http://www.ams.org/mathscinet-getitem?mr=#1}{#2}
}
\providecommand{\href}[2]{#2}

\end{document}